\documentclass[reqno,11pt]{amsart}

\usepackage{amsmath}
\usepackage{amsthm}
\usepackage{amssymb}
\usepackage{amscd}
\usepackage{xypic}
\usepackage{verbatim}
\usepackage{graphicx}
\usepackage{palatino}
\usepackage{bbm}

\usepackage[plainpages=false,colorlinks,hyperindex,pdfpagemode=None,bookmarksopen,linkcolor=red,citecolor=blue,urlcolor=blue]{hyperref}
\usepackage{pdflscape}
\usepackage{stmaryrd}

\usepackage{multirow}

\DeclareFontFamily{U}{mathx}{\hyphenchar\font45}
\DeclareFontShape{U}{mathx}{m}{n}{
      <5> <6> <7> <8> <9> <10>
      <10.95> <12> <14.4> <17.28> <20.74> <24.88>
      mathx10
      }{}
\DeclareSymbolFont{mathx}{U}{mathx}{m}{n}
\DeclareFontSubstitution{U}{mathx}{m}{n}
\DeclareMathAccent{\widecheck}{0}{mathx}{"71}
\DeclareMathAccent{\wideparen}{0}{mathx}{"75}

\theoremstyle{plain}
\newtheorem{theorem}{Theorem}[section]

\newtheorem*{theorem*}{Theorem}
\newtheorem{proposition}{Proposition}[section]
\newtheorem*{proposition*}{Proposition}
\newtheorem{lemma}{Lemma}[section]
\newtheorem*{lemma*}{Lemma}

\newtheorem{corollary}{Corollary}[section]
\newtheorem*{corollary*}{Corollary}

\theoremstyle{definition}

\theoremstyle{remark}
\newtheorem{remark}{Remark}[section]

\renewcommand{\comment}[1] {  }

\DeclareFontFamily{OT1}{rsfs}{}
\DeclareFontShape{OT1}{rsfs}{n}{it}{<-> rsfs10}{}
\DeclareMathAlphabet{\mathscr}{OT1}{rsfs}{n}{it}

\newcommand{\X}{\mathfrak{X}}
\newcommand{\Y}{\mathfrak{Y}}

\newcommand{\Ad}{\mathrm{Ad}}

\newcommand{\C}{\mathfrak{C}}

\newcommand{\adele}{{\mathbb{A}_k}}

\newcommand{\CC}{\mathbb{C}}

\newcommand{\RR}{\mathbb{R}}

\newcommand{\Ind}{\operatorname{Ind}}

\newcommand{\Aut}{{\operatorname{Aut}}}

\newcommand{\Gm}{\mathbb{G}_m}
\newcommand{\Ga}{\mathbb{G}_a}

\newcommand{\GL}{\operatorname{GL}}
\newcommand{\Mat}{\operatorname{Mat}}
\newcommand{\Sym}{\operatorname{Sym}}

\newcommand{\PGL}{\operatorname{PGL}}

\newcommand{\SL}{\operatorname{SL}}

\newcommand{\tr}{\operatorname{tr}}

\newcommand{\spec}{\operatorname{spec}}

\newcommand{\diag}{{\operatorname{diag}}}

\newcommand{\ev}{\operatorname{ev}}

\newcommand{\gr}{{\operatorname{gr}}}

\newcommand{\Id}{\operatorname{Id}}

\newcommand{\Std}{{\operatorname{Std}}}

\newcommand{\val}{{\operatorname{val}}}

\newcommand{\TF}{{\operatorname{TF}}}
\newcommand{\RTF}{{\operatorname{RTF}}}

\newcommand{\LX}{{^LX}}
\newcommand{\LY}{{^LY}}
\newcommand{\LG}{{^LG}}
\newcommand{\LH}{{^LH}}

\newcommand{\Dfrac}[2]{%
  \ooalign{%
    $\genfrac{}{}{1.2pt}0{#1}{#2}$\cr%
    $\color{white}\genfrac{}{}{.4pt}0{\phantom{#1}}{\phantom{#2}}$}%
}

\begin{document}

\swapnumbers

\numberwithin{equation}{section}
\setcounter{tocdepth}{1}
\title{Relative functoriality and functional equations via trace formulas}
\author{Yiannis Sakellaridis}
\email{sakellar@rutgers.edu}

\address{Department of Mathematics and Computer Science, Rutgers University at Newark, 101 Warren Street, Smith Hall 216, Newark, NJ 07102, USA.}


\begin{abstract}
Langlands' functoriality principle predicts deep relations between the local and automorphic spectra of different reductive groups. This has been generalized by the relative Langlands program to include spherical varieties, among which reductive groups are special cases. In the philosophy of Langlands' ``beyond endoscopy'' program, these relations should be expressed as comparisons between different trace formulas, with the insertion of appropriate $L$-functions. The insertion of $L$-functions calls for one more goal to be achieved: the study of their functional equations via trace formulas. 

The goal of this article is to demonstrate this program through examples, indicating a local-to-global approach as in the project of endoscopy. Here, scalar transfer factors are replaced by ``transfer operators'' or ``Hankel transforms'' which are nice enough (typically, expressible in terms of usual Fourier transforms) that they can be used, in principle, to prove global comparisons (in the form of Poisson summation formulas). Some of these examples have already appeared in the literature; for others, the proofs will appear elsewhere.
\end{abstract}

\maketitle

\tableofcontents

\section{Introduction}

\subsection{Relative functoriality and ``beyond endoscopy''}
Let $G$ be a reductive group over a global field $k$, that is: a number field, or the function field of a curve over a finite field. We will be denoting the ring of adeles of $k$ by $\adele$, and the automorphic quotient space $G(k)\backslash G(\adele)$ by $[G]$.

The Langlands program is about \emph{automorphic representations}: those are, essentially, the irreducible representations that appear in the Plancherel decomposition of $L^2([G])$, considered as a $G(\adele)$-representation. There are two main conjectures in the Langlands program: One of them, \emph{reciprocity}, aims to attach automorphic representations to representations of the Galois group of $k$ and, more generally, to motives over $k$ --- and vice versa, for certain automorphic representations that are considered ``algebraic'' \cite{BG}. 
More precisely, the Galois representations should have image in the $L$-group $\LG$ of $G$. The other main conjecture, \emph{functoriality}, aims to relate automorphic representations of two different groups $G_1$ and $G_2$ for every homomorphism $\LG_1\to \LG_2$ of $L$-groups. This paper is related to the latter conjecture.

It has become clear through the decades and, especially, through the work of Arthur \cite{Arthur-unipotent}, that the ``automorphic spectrum'' of a group $G$ is not a very convenient object to study. Instead, it is preferable to study the automorphic spectrum of its trace formulas, and more precisely the individual summands of the trace formulas obtained through the process known as pre-stabilization. For example, the project of \emph{endoscopy} involved a comparison the various endoscopic summands of the twisted trace formula of $\GL_n$, on one hand, and the stable trace formula of classical groups, on the other --- leading to a resolution of the functoriality conjecture for $G_1=$ a classical group, and $\LG_1\to \LG_2 =\GL_N$ the standard representation of its $L$-group \cite{Arthur, Mok, KMSW}.

It has also become clear through the work of Jacquet, D.\ Prasad and others that the functoriality conjecture is not restricted to reductive groups, but to a more general class of $G$-spaces called \emph{spherical varieties}. This class includes symmetric spaces, and any reductive group $H$ can be viewed as a symmetric space under the action of $G=H\times H$ by left and right multiplication. Those varieties should have $L$-groups (although, for now, only their connected components have been defined in satisfactory generality, \cite{KnSch}), and any morphism $\LX_1 \to \LX_2$ between the $L$-groups of two varieties should give rise to relations between their \emph{automorphic spectra} \cite{SV}.

Here, again, the natural setting to define automorphic spectra is that of a generalization of the trace formula, namely the \emph{relative trace formula} of Jacquet. In modern language, the relative trace formula is a distribution on the adelic points of a quotient stack $\X = [(X\times X')/G^\diag]$ \cite{SaStacks}, where $X, X'$ are two spherical varieties (possibly the same) for the same group $G$, together with two ways to write this distribution, one ``geometric'' and one ``spectral''. The content of the spectral side is, by definition, the ``automorphic spectrum'' of the pair $(X,X')$ (or of the space $X$, if $X'=X$). The usual trace formula is obtained by setting $X=X'=H$ and $G=H\times H$, in which case $\X = \frac{H}{H}$, where we use the fraction notation to denote the quotient of the numerator under the action of the denominator by conjugation.

There are cases where such a functorial transfer is easy to establish. Take, for example, $X_1$ to be the space $\Gm\backslash \PGL_2$ and let $X_2$ stand for the Whittaker model of $\PGL_2$. (This is the space $N\backslash \PGL_2$, where $N\simeq \Ga$ is a unipotent subgroup, endowed with a non-trivial adele class character $\psi$ of $N$; such cases can be included into the above general setting.) The dual group is $\SL_2$ for both $X_1$ and $X_2$, so we expect a local and global ``functorial transfer'' between them, that should match the spectral sides of their Plancherel and relative trace formulas. In this case, Hecke's ``unfolding'' method shows that, globally, the $\Gm$-period of a cusp form $\varphi$ can be related to the Whittaker/Fourier period, as follows:
$$ \int_{[\Gm]} \varphi\left(\begin{pmatrix}
                        a \\ & 1 
                       \end{pmatrix}\right) |a|^s da = \int_{\mathbb A_k^\times} \int_{[N]} \varphi\left(\begin{pmatrix}
                        a \\ & 1 
                       \end{pmatrix} n\right) \psi(n) dn |a|^s da.$$
If desired, this can be interpreted in terms of the spectral sides of the associated relative trace formulas. Over a local field $F$, it was observed in \cite[\S 9.5]{SV} that Hecke's method gives rise to an equivariant isometry:
\begin{equation}\label{unfoldingL2} L^2(\Gm \backslash \PGL_2(F)) \xrightarrow{\sim} L^2 (N\backslash \PGL_2(F), \psi).
\end{equation}
Thus, the local and global spectra of the spaces $X_1$ and $X_2$ match. 

The problem is that such simple-minded methods (and others, less simple-minded, such as the theta correspondence), which produce spectral identities and period identities at the level of the spaces under consideration, while very powerful for a class of examples, can only cover a tiny portion of the conjectural cases of functoriality. Clearly, something much broader and deeper is waiting to be discovered.

The impending completion of the endoscopy program led Langlands to formulate a strategy for proceeding further towards the functoriality conjecture; this strategy bears the name ``beyond endoscopy'' \cite{Langlands-BE}. The basic goal of this program is to construct a comparison between the \emph{stable trace formulas} of two reductive groups $G_1, G_2$ for any morphism $\LG_1\to \LG_2$ between their $L$-groups. Unfortunately, many aspects of such a comparison remain unclear, and even the simplest instances of its implementation, in the setting of the trace formula as envisioned by Langlands, are technically very demanding. 

Locally, for example, such a morphism of $L$-groups should induce a map 
\begin{equation}\label{Tgroups} \mathcal T: \mathcal S(\frac{G_2}{G_2}) \to \mathcal S(\frac{G_1}{G_1})
\end{equation}
between spaces of stable orbital integrals (the notation will be explained below), whose adjoint will induce the map of (tempered) stable characters envisioned by functoriality. 
I believe that understanding the nature of this local comparison is key for the ``beyond endoscopy'' program, and will also shed light on the global comparison of stable trace formulas. Indeed, while in the case of endoscopy there was a natural matching between (stable) conjugacy classes, leading to the conjecture that orbital integrals for the two groups should match up to (explicit) transfer factors, such a matching is missing in the case of ``beyond endoscopy'', and therefore the comparison will rely on understanding the \emph{transfer operators} $\mathcal T$ of \eqref{Tgroups}. There have been efforts to reverse-engineer these operators, in specific cases, from known character formulas \cite{Langlands-ST, Johnstone}, and these may eventually shed some light on the question. At this point, it is not clear how the formulas of Langlands (going back to Gelfand and Graev) and Johnstone could be used to obtain a global comparison of trace formulas.

If one is willing to look at the more general problem of ``relative'' functoriality, there are a lot more, and easier, examples of comparisons where these transfer operators can be studied. For example, two spherical varieties $X_1$ and $X_2$ (one of which could be a group) could have the same $L$-group:  
$$ \LX_1 \xrightarrow\sim \LX_2,$$
as we saw above in the example of $X_1 = \Gm\backslash \PGL_2$ and $X_2= (N,\psi)\backslash \PGL_2$. While, in general, there will not be a simple-minded way such as \eqref{unfoldingL2} to directly compare the spectra of the two spaces, there may be a way to perform such a comparison at the level of relative trace formulas.  Establishing the corresponding functorial transfer between the local and global spectra of two spaces with $\LX_1\simeq \LX_2$
can be a highly non-trivial problem that, if solved in generality, will on one hand provide significant insights into the ``beyond endoscopy'' program and, on the other, lead to important relations between periods of periods automorphic forms and special values of $L$-functions, such as the global Gan--Gross--Prasad conjectures \cite{II}.

\vspace{6pt}

The goal of this paper is to demonstrate the nature of these transfer operators in several examples, mostly in the setting when $\LX_1 = \LX_2$.

\vspace{6pt}

Incidentally, although not working with trace formulas, the identity map between $L$-groups features in the work of L.\ Lafforgue on automorphic kernels \cite{Lafforgue}, where local kernels are described for a ``model transition'' from the Whittaker model $X_1=(N,\psi)\backslash \GL_n$ (where $N\subset \GL_n$ is a maximal unipotent subgroup, and $\psi$ a non-degenerate characters) to the group $X_2=H=\GL_n$. Generalizing the formalism of converse theorems, Lafforgue explains that the automorphicity of the global kernel relies on functional equations for certain $L$-functions. This approach seems unorthodox from Langlands' point of view: properties of $L$-functions should be consequences, not prerequisites for functoriality. However, as Ng\^o's recent retake on the program of Braverman and Kazhdan shows \cite{Ngo-monoids} (related also to the author's treatment of the Rankin--Selberg method using spherical varieties \cite{SaRS}), the functional equation of $L$-functions is not necessarily very far from the philosophy of trace formulas and ``beyond endoscopy''. Indeed, one can attempt to recast the functional equation as a Poisson summation formula for a conjectural Fourier transform between ``Schwartz spaces'' for dual reductive monoids (or more general spherical varieties) --- generalizing Tate's thesis \cite{Tate-thesis} and the work of Godement and Jacquet \cite{GJ}. Moreover, instead of studying the Fourier transform between these Schwartz spaces, one can study its descent (the ``Hankel transform'') to their spaces of orbital integrals \cite{Ngo-Hankel}. The Poisson summation formula, then, becomes another instance of a non-endoscopic comparison between trace formulas. Such a derivation of the functional equation using a relative trace formula has already appeared in papers of P.E.\ Herman and the author \cite{Herman, SaBE2}.

\subsection{Contents of the paper} The goal of the present paper is to examine the possibility of non-standard comparisons of trace formulas which correspond to instances of functoriality or of functional equations of $L$-functions. This results to a collection of examples of local ``transfer operators'' and ``Hankel transforms'' between the spaces of test functions (or, more correctly, measures) of relative trace formulas, mostly in rank one. The important conclusion is that these local transfer operators are given in terms of abelian Fourier transforms and multiplication with innocuous scalar factors that are suitable, in principle, for a global comparison of trace formulas, that will have the form of a Poisson summation formula.  I do not discuss this global comparison here; however, the methods employed in \cite{SaBE2} should generalize to all cases discussed here. Another important finding is that these transfer operators behave well under ``degeneration'': that is, when the spherical varieties under consideration are deformed into their ``boundary degenerations'' (such as: $H=\SL_2$ degenerating to the $H\times H$-variety of $2\times 2$-matrices of determinant zero), then the transfer operators degenerate to the corresponding operators for the degenerations (which can be understood using abelian Fourier theory). The examples presented here lead to a more systematic understanding of transfer operators, at least in rank one, that will be the object of an upcoming paper \cite{SaRankone}. 

I start by explaining the general formalism for these comparisons, in a slightly restricted setup that will be sufficient for the examples considered in this paper.

\subsubsection{Schwartz spaces} \label{ssslanguage} Let $\X$ denote a quotient stack of the form $[(X\times X)/G]$, where $X$ is a smooth spherical $G$-variety, and the $G$-action on the product $X\times X$ is understood to be the diagonal one. We will often use the isomorphism $\X = [H\backslash G/H]$, when $X$ is a homogeneous space $H\backslash G$. Everything is defined over a local field $F$, keeping in mind that in the end we would like to see integral transforms that satisfy some sort of ``Poisson summation formula'' over a global field. Whenever there is no confusion, we will denote the set $Y(F)$ of $F$-points simply by $Y$, for any variety (or stack) $Y$.

The stack $\X$ (rather, its $F$-points, which can be thought of as a ``Nash stack over $F$'', s.\ \cite{SaStacks}) may come equipped with a complex line bundle; the only example where we will need a non-trivial line bundle is for the quotient $[N\backslash G/N] = [(N\backslash G \times N\backslash G)/G^\diag]$, where $N$ is a maximal unipotent subgroup of $G$, and the line bundle is defined by a non-degenerate character $\psi$ of $N$. In those cases, we will be writing $\X = [(N,\psi)\backslash G/(N,\psi)]$.

Let $\mathcal S(X)$ denote the space of Schwartz measures on $X=X(F)$ (or measures valued in the line bundle given). These are smooth measures which decay rapidly, together with their polynomial derivatives, in the Archimedean case, and smooth, compactly supported measures in the non-Archimedean case. In some instances where we will need to work with Schwartz functions, instead of measures, we will denote the space of those by $\mathscr F(X)$. Of course, if $dx$ is a nowhere vanishing smooth measure of polynomial growth (such as a Haar measure), we have $\mathcal S(X) = \mathscr F(X) dx$. 

Our varieties will always be quasi-affine, and we denote the categorical quotient $(X\times X)\sslash G = F[X\times X]^G$ by $\C_X$. Thus, we have a push-forward of measures
\begin{equation}\label{pushforward} \mathcal S(X\times X) \to \mathcal M(\C_X),\end{equation}
where $\mathcal M$ denotes the space of all measures. When $X$ is endowed with a non-trivial line bundle, we need to make a choice of trivialization in order for the push-forwards to be defined as scalar-valued measures on $\C_X$. For $N\backslash \PGL_2/N$, with $N$ the subgroup of upper triangular unipotent matrices, the open Bruhat cell is a union of cosets $N\tilde\xi N$, where
$$\tilde \xi = \begin{pmatrix} & -1 \\ \xi \end{pmatrix},\,\, \xi\in \Gm,$$
and we trivialize measures which satisfy $\mu(n_1 g n_2) = \psi(n_1)\mu(g) \psi(n_2)$ (over the open Bruhat cell) by ``evaluating'' them at the elements $\tilde\xi$. 
For $\SL_2$, the corresponding representatives will be 
$$ \tilde\zeta = \begin{pmatrix} & -\zeta^{-1} \\ \zeta \end{pmatrix},\,\, \zeta\in \Gm,$$

Assume that, at the level of (isomorphism classes of) $F$-points, the map $X(F)\times X(F)\to \X(F)$ is surjective; this is the case, for example, for $N\backslash G/N$ and for $T\backslash \PGL_2/ T$, when $T$ is a split torus, but not if $T$ is a non-split torus. Under this assumption,\footnote{If the map $X(F)\times X(F)\to \X(F)$ is not surjective on points, this means that there are non-trivial $G$-torsors $R$ which contribute $F$-points to $\X$. Then, \eqref{pushforward} should be replaced by 
$$ \bigoplus_\alpha \mathcal S(X^\alpha\times X^\alpha) \to \mathcal M(\C_X),$$
with $\alpha$ ranging over all isomorphism classes of $G$-torsors $R^\alpha$, and $X^\alpha:= X\times^G R^\alpha$, a $G^\alpha$-space, where $G^\alpha$ is the inner form $\Aut^G(R^\alpha)$. For example, for $X=T\backslash\PGL_2$, $X^\alpha$ includes $X$ as well as the quotient $T\backslash PD^\times$, where $D$ is the quaternion division algebras over $F$, and $PD^\times$ is the quotient of its multiplicative group by the center. For simplicity of exposition, we will ignore such cases in this paper, although the particular case of a non-split torus in $\PGL_2$ has already been studied in \cite{SaBE1,SaBE2}.} the space of measures that we obtain as the image of \eqref{pushforward} will be denoted by $\mathcal S(X\times X/G)$ or $\mathcal S(\X)$, by abuse of notation. (This notation was used in \cite{SaStacks} for the space $\mathcal S(X\times X)_G$ of $G$-coinvariants; in our examples, the push-forwards of measures to $\C_X$ will correspond to \emph{stable} coinvariants. These are the ``test measures'' for \emph{stable} trace formulas.)

We will also need to consider extended Schwartz spaces of measures; those are spaces of smooth measures with specified behavior at a certain ``infinity''. The conditions at infinity will depend on the situation considered, but these conditions will not always appear explicitly in the notation; rather, any such ``non-standard'' space of Schwartz measures will be denoted by exponents $\pm$ (where we will use both signs when we want to indicate different directions ``at infinity''). Thus, for example, to insert certain $L$-functions into the Kuznetsov formula we will be working with certain non-standard spaces of Whittaker measures, denoted $\mathcal S^-((N,\psi)\backslash G)$, which vanish (or are of rapid decay) close to the ``cusp'', but have non-trivial asymptotics close to the ``funnel'', cf.\ \S \ref{ssnonstandard}. (If we think of $N\backslash \PGL_2$ or $N\backslash \SL_2$ as a cone, resp.\ affine plane, without the origin, the ``cusp'' is then the origin, and the ``funnel'' is the other direction towards infinity.)  Similarly, the theory of asymptotics of (generalized) matrix coefficients on spherical varieties gives rise to functions (or measures) on horospherical spaces that have non-trivial asymptotics ``close to the cusp'', and those will be denoted by the exponent $+$, e.g., 
$$\mathcal S((N,\psi)\backslash G) \to \mathcal S^+(N\backslash G).$$
Correspondingly, the push-forwards of such measures to the categorical quotient $\C_X$ (where $X$ is the space under consideration) will be denoted by $\mathcal S^\pm (\X)$. For specific choices of these extended Schwartz spaces that will concern us in this paper, see \S \ref{ssnonstandard}.

\subsubsection{The relative trace formula}

This paper is concerned with \emph{local} comparisons (``transfer operators'') between spaces of Schwartz measures, but these comparisons are motivated by the desire to have \emph{global} comparisons between relative trace formulas, so let us say a few things about the relative trace formula. The presentation that   follows is quite naive, pretending that all integrals converge. This will rarely be true; instead, integrals have to be regularized. A general process of regularization was explained in \cite[\S 5--6]{SaStacks}, but it does not cover all cases (notably, it does not cover the Arthur--Selberg trace formula). Thus, a relative trace formula in the generality presented here has not been developed yet.

Let $k$ be a global field, $\adele$ its ring of adeles. Assume that $X$ is a smooth, quasi-affine $G$-variety defined over $k$. Having defined the Schwartz spaces of $X$ over any local completion $k_v$ of $k$, we now endow them with a \emph{basic vector} at (almost all) non-Archimedean places, which is typically the characteristic function of $X(\mathfrak o_v)$ (where $\mathfrak o_v$ is the ring of integers of $k_v$, and we assume an integral model outside of a finite set of places), times a $G$-invariant measure (which should be chosen to factorize a canonical global choice of measure, the Tamagawa measure). We keep allowing the case $X=N\backslash G$, endowed with an adele class character $\psi$ of $N(\adele)$. In the case of ``extended Schwartz spaces'' mentioned above, this basic vector will need to be modified --- I point again to \S \ref{ssnonstandard}. Having the basic vectors, we can define global Schwartz spaces as restricted tensor products with respect to these vectors:
$$ \mathcal S(X(\adele)) = \bigotimes_v' \mathcal S(X(k_v)).$$

Correspondingly, for $\X = [X\times X/G]$ we define the basic vector of $\mathcal S(\X(k_v))$ as the push-forward of the basic vector of $\mathcal S(X\times X(k_v))$, and the  global Schwartz space $\mathcal S(\X(\adele))$ as the corresponding restricted tensor product over all places.

The \emph{relative trace formula} of $\X$ is a global analog of the local \emph{Plancherel formula}. The local Plancherel formula is a decomposition of the inner product of two measures (fixing an invariant measure $dx$ on $X(F)$) in terms of \emph{local relative characters}:

\begin{equation} \int_{X(F)} \frac{\Phi_1(x) \Phi_2(x)}{dx} = \int_{\widehat{G}} J_\pi(\Phi_1\otimes \Phi_2) \mu_X(\pi).\end{equation}
We recall that a \emph{relative character} $J_\pi$ on $X=X(F)$, associated to an admissible representation $\pi$, is a $G$-invariant functional on $\mathcal S(X\times X)$ that factors as
$$ \mathcal S(X\times X)\to \pi\hat\otimes \tilde\pi \xrightarrow{{\left<\,\, , \,\, \right>}} \CC.$$
In the group case ($X=H$, $G=H\times H$), this is the same as a character, if we identify the space of $G$-invariant generalized functions on $H\times H$ with conjugation-invariant generalized functions on $H$.

Now, globally, we define a distribution that resembles the local inner product of two functions/measures on $X$, but factors through the space of automorphic functions:

$$\xymatrix{
\Phi_1 \in \mathcal S(X(\adele)) \ar[r] & \Sigma\Phi_1(g):= \sum_{\gamma\in X(k)} \Phi_1(\gamma g) \in C^\infty([G])  \\
&  \bigotimes \ar[r]^{\left<\,\, , \,\,\right>_{[G]}} &\CC\\
\Phi_2 \in \mathcal S(X(\adele)) \ar[r] & \Sigma\Phi_2(g):= \sum_{\gamma\in X(k)} \Phi_2(\gamma g) \in C^\infty([G])  
}$$

Here, the notation is that $[G]=G(k)\backslash G(\adele)$, and we feel free (globally) to identify measures with functions, by dividing by a Tamagawa measure. The last arrow is the integral $\int_{[G]} \Sigma\Phi_1\cdot \Sigma\Phi_2$ of the ``theta series'' $\Sigma \Phi_1$, $\Sigma\Phi_2$, assuming that it converges. The resulting functional 
$$ \mathcal S(X(\adele)) \hat\otimes \mathcal S(X(\adele)) \to \CC$$ will be denoted by $\RTF_X$ or $\RTF_{\X}$, for ``relative trace formula''. The relative trace formula itself is, actually, the spectral and geometric decomposition of this functional; the former in terms of relative characters associated to automorphic representations of $G$, and the latter, roughly, in terms of $G(k)$-orbits on $X\times X(k)$.

More precisely (but very naively, as far as convergence issues are concerned), by using the Plancherel formula for $L^2([G])$ we can decompose $\left<\Sigma\Phi_1,\Sigma\Phi_2\right>$ into a direct integral, over the set of automorphic representations of $G$, of \emph{global relative characters}:
$$ J_\pi^{\rm{gl}}(\Phi_1\otimes\Phi_2) = \sum_{\varphi} \left(\int_{[G]} \Sigma\Phi_1\cdot \bar\varphi\right) \left(\int_{[G]} \Sigma\Phi_2 \cdot\varphi\right),$$
where the sum ranges over an orthonormal basis of $\pi$. For $X=H\backslash G$ homogeneous, these global relative characters can also be expressed in terms of ``period integrals'' of automorphic forms over $[H]$, and those are often (actually, always in the multiplicity--free case) \emph{Eulerian}, i.e., pure tensors:
$$ J_\pi^{\rm{gl}} = \bigotimes_v J_{\pi_v}$$
of local relative characters, when $\pi$ is decomposed as $\bigotimes_v' \pi_v$. A precise expression for the local factors $J_{\pi_v}$ is provided by the Ichino--Ikeda conjecture \cite{II}, and its generalization described in \cite[\S 17]{SV}. The important point here is that this Euler factorization results in \emph{special values of $L$-functions} appearing on the spectral side of the relative trace formula, $L$-values that depend on the variety $X$. I point the reader to the aforementioned references for details on the general case; in this paper, we will only discuss specific examples that will feature in our comparisons.

The geometric decomposition of the relative trace formula or, at least, of its ``stable'' part (which is often the whole $\RTF$), is expressed in terms of the spaces $\mathcal S(\X(k_v))$ of ``stable orbital integrals'' of the various completions of $k$. I point the reader to \cite[\S 6.4]{SaStacks} for an attempt at a general formulation, where the relative trace formula is presented as a sum of the form:
\begin{equation}\label{RTFonstack}
 \sum_{\xi \in \X^{\rm{ss}}(k)} \ev_\xi,
\end{equation}
of ``evaluations'' at the ``semisimple'' global points of the stack $\X$.
Here, we will be concerned with ways to ``match'' the local Schwartz spaces for different quotients of the form $[X\times X/G]$.

\subsubsection{Comparisons} Now let $\X$, $\Y$ be two such quotients, corresponding to spherical varieties $X, Y$ (possibly for different groups $G, G'$, and possibly endowed with line bundles as before). The functoriality predicted by a homomorphism 
\begin{equation}\label{mapLgroups}\LY\to\LX
\end{equation}
 between their $L$-groups should be realized, locally, by a linear ``transfer'' map between the Schwartz spaces
$$ \mathcal T: \mathcal S(\X) \to \mathcal S(\Y),$$
where $\X = [X\times X/G]$, $\Y=[Y\times Y/G']$, satisfying certain local and global desiderata, that we now describe.

The transfer operator $\mathcal T$ should be such that,
locally, the pull-backs of stable tempered relative characters on $\Y$ should be stable tempered relative characters on $\X$, corresponding to the map of $L$-packets that should be associated to \eqref{mapLgroups}.  In the cases that we consider, for an irreducible $\pi$ the space of morphisms $\mathcal S(X\times X)\to \pi\hat\otimes \tilde\pi $ is at most one-dimensional, so the relative character for an irreducible representation is determined uniquely up to scalar. Unlike the group case, there is no canonical normalization of this scalar, but the local factor $J_\pi$ of the Ichino--Ikeda conjecture and its generalizations provides a distinguished such normalization. This normalization will be discussed individually in some examples presented in this paper; for the general case, I point the reader to Chapter 17 of \cite{SV}.

We should also have a ``fundamental lemma'', stating that the image of the basic vector of $\mathcal S(\X)$ under the transfer is the basic vector of $\mathcal S(\Y)$ (at non-Archimedean places), so that the transfer operator translates to a transfer between global Schwartz spaces: 
$$ \mathcal T: \mathcal S(\X(\adele)) \to \mathcal S(\Y(\adele)),$$
where $\adele$ denotes the ring of adeles of $k$.

Globally, when \eqref{mapLgroups} is an isomorphism, the map should commute with the (stable) ``relative trace formula'' distributions, thus we should have a commutative diagram:
$$ \xymatrix{ \mathcal S(\X(\adele)) \ar[rr]^{\mathcal T}\ar[dr]_{\RTF_\X} & & \ar[dl]^{\RTF_\Y}\mathcal S(\Y(\adele))\\
& \CC &}.$$

(When the map of $L$-groups is not an isomorphism, it is not clear what to ask of the corresponding relative trace formulas; the thesis of Venkatesh \cite{Venkatesh}, not discussed here but revisited in the upcoming paper \cite{SaTransfer2}, gives some indication about what such a global comparison might look like.)

These desiderata are actually incompatible with each other; for example, as mentioned in the previous subsection, relative trace formulas include periods of automorphic forms which very often are equal to special values of $L$-functions, and those $L$-functions typically do not coincide for $\X$ and $\Y$. Thus, they have to be inserted into the trace formulas, which is why the transfer operator should, typically, be between non-standard spaces of test measures
\begin{equation}\label{transfer}
 \mathcal T: \mathcal S^\pm(\X) \to \mathcal S^\pm(\Y)
\end{equation}
(at least for one of the quotients involved). This includes the case when we want to use trace formulas to encode the functional equation of $L$-functions; in that case, the spectral decomposition of the left hand side will involve some $L$-functions $L(\pi,\rho)$, and that of the right hand side will involve the dual $L$-functions $L(\pi,\rho^\vee)$; the transfer operator would be a descent, to the level of orbital integrals, of the ``Fourier transforms'' envisioned by Braverman and Kazhdan \cite{BK1, BK2}. Following Ng\^o \cite{Ngo-Hankel}, we call ``Hankel transform'' the descent of such a Fourier transform to the level of orbital integrals.

The basic questions to be answered in such a setting are the following:

\begin{itemize}
 \item Can we describe the appropriate Schwartz spaces $\mathcal S^\pm(\X)$ $\mathcal S^\pm(\Y)$, and the correct transfer operator between them, that is, a linear isomorphism that pulls back relative characters to relative characters, and satisfies the fundamental lemma? In practice, checking the statement on relative characters locally can be difficult or impossible; instead, one might just require that the statement be true for unramified representations, in the form of a ``fundamental lemma for the Hecke algebra'', and deduce it for the rest by global-to-local means. 
 \item In the case $\LX=\LY$, globally, does the transfer operator $\mathcal T$ commute with the ``relative trace formula'' functionals? (This would be a form of Poisson summation, expressed in terms of a ``Poisson sum'' of the form \eqref{RTFonstack}.) In the general case, what is the nature of the comparison of relative trace formulas afforded by the transfer operator? 
\end{itemize}

\subsubsection{List of cases considered in this paper}

In this paper, we will revisit and cast in the setup presented above the following cases, most of which have appeared in the literature in one form or another. New results are presented without proof; proofs appear in the papers \cite{SaTransfer1, SaTransfer2}.

For the examples below, when nothing is mentioned the map between $L$-groups corresponding to our comparison is an isomorphism. The numbering corresponds to the chapters of this paper.

The first group of comparisons concerns functoriality between different spaces:

\begin{description}
 \item[[\ref{sec:Cartan}]] $\X = $ the adjoint quotient of any quasi-split group $H$, and $\Y = $ its universal Cartan. The dual $\LY$ is the canonical maximal torus of $\LH$. This is a warm-up exercise.
 \item[[\ref{sec:degen}]]  $\X = (X\times X)/G$ for any spherical $G$-variety $X$ (we will assume $G$ to be split, for simplicity), and $\Y = (X_\emptyset \times X_\emptyset)/G$, where $X_\emptyset$ is its (most degenerate) boundary degeneration. This is very closely related to the previous case, since the boundary degeneration is parabolically induced from the universal Cartan $A_X$ of $X$; the pertinent map of $L$-groups is the canonical embeding ${^LA_X} \hookrightarrow \LX$. We will not give a geometric expression for transfer operators here, but we will perform some calculations of relative characters and transfer operators between degenerations of different spaces, that will later be compared to transfer operators between the original spaces.
 \item[[\ref{sec:Rudnick}]]  $\X= (N,\psi)\backslash \SL_2 /(N,\psi)$ and $\Y = \frac{\SL_2}{\SL_2}$, i.e., comparing the Kuznetsov and the stable Selberg trace formula for the group $\SL_2$. This comparison was performed, in a classical language, in Rudnick's thesis \cite{Rudnick}; here we will only discuss the local aspects.
 \item[[\ref{sec:Waldspurger}]] $\X = (N,\psi)\backslash\PGL_2/(N,\psi)$ and $\Y = T\backslash \PGL_2/T$, where $T$ is a torus, taken (for simplicity) to be split. This is a review of \cite{SaBE1}. 
\end{description}

The second group has to do with functional equations of $L$-functions:

\begin{description}
 \item[[\ref{sec:standard}]] $\X = \Y = (N,\psi)\backslash \GL_n / (N,\psi)$, but equipped with non-standard test measures, corresponding to the standard representation of the $L$-group (for $\X$) and its dual (for $\Y$). The Hankel transforms, here, which correspond to the functional equation of the standard $L$-function, were computed by Jacquet \cite{Jacquet}.
 \item[[\ref{sec:sym2}]] $\X =\Y = (N,\psi)\backslash \SL_2 \times \Gm / (N,\psi)$, equipped with non-standard test measures, corresponding to the symmetric square $L$-function (for $\X$) and its dual (for $\Y$). This case has not yet appeared in the literature, but in \cite{SaTransfer2} we compute its Hankel transform out of the Rankin--Selberg method.
\end{description}

In the cases of Hankel transforms, it will turn out that we obtain nicer formulas in these cases if we work with \emph{half-densities}, instead of measures, cf.\ \S \ref{sec:standard}, \S \ref{sec:sym2}.

In all cases, we will not only compute the transfer operators, but we will also study how those degenerate, by which I mean the following: For every spherical variety $X$, we have its ``boundary degeneration'' or ``asymptotic cone'' $X_\emptyset$, which has already been mentioned. For example, the boundary degeneration of the $G=\SL_2\times \SL_2$-variety $X=\SL_2$ is the $G$-variety of non-zero $2\times 2$ matrices of determinant zero. In the non-Archimedean case, we have (presently, under some assumptions) a universal ``asymptotics'' map 
\begin{equation}\label{Bmap}e_\emptyset^*: \mathcal S(X) \to \mathcal S^+(X_\emptyset),\end{equation}
landing in a space of smooth measures whose support has compact closure in an affine embedding of $X_\emptyset$. Moreover, there is a $G\times \Gm$-family $\mathcal X\to \Ga$, whose general fiber is isomorphic to $X$ and whose special fiber is isomorphic to $X_\emptyset$. (When $X= (N,\psi)\backslash G$, then $X_\emptyset$ is equal to $X$ as a variety, but this family corresponds to the degeneration of the character $\psi$ to the trivial character.)

For a pair of spaces $X, Y$ (and the related stacks $\X, \Y$), it is natural to attempt a comparison between the transfer operator for $\X, \Y$, and the transfer operator for (the corresponding quotients associated to the boundary degenerations) $\X_\emptyset$, $\Y_\emptyset$. Calculations of relative characters will allow us to determine what the transfer operator or Hankel transform should be in each case for the degenerations. 

It turns out, in all cases considered, that the transfer operators are simply \emph{deformations} of the corresponding tranfer operators for the boundary degenerations. Thus, the examples discussed in this paper reveal a lot of structure behind the transfer operators of ``beyond endoscopy''; this structure was previously unnoticed in the --- heavily tilted towards analytic number theory and overlooking local aspects --- literature in the subject.

\subsection{Acknowledgements} 

Most of the calculations in this paper were performed while visiting the University of Chicago during the winter and spring quarters of 2017. I am grateful to Ng\^o Bao Ch\^au for the invitation, and for numerous conversations and references which made this paper possible. His ideas permeate the paper. The paper was finished during my stay at the Institute for Advanced Study in the Fall of 2017.

This work was supported by NSF grant DMS-1502270, and by a stipend to the IAS from the Charles Simonyi Endowment.

\part{General considerations}

\section{Calculus of equivariant Fourier transforms} \label{sec:Fourier}

Before we embark on trace formula comparisons, we need to introduce a basic family of transformations that will appear all the time, the \emph{equivariant Fourier transforms}. Here, ``equivariant'' means with respect to the action of a multiplicative group or, more generally, a torus. 

We fix throughout a non-trivial unitary character $\psi$ of the additive group $F$. At non-Archimedean places, we will be assuming that its conductor is the ring of integers $\mathfrak o$. We also fix a measure $dx$ on $F$ which is self-dual with respect to $\psi$; this induces a measure $|\omega|$ on $X(F)$, for every volume form $\omega$ on a smooth variety $X$ over $F$. We will also use the measure $d^\times x := \frac{dx}{|x|}$ on the multiplicative group $F^\times$.

Let $T$ be a torus, and $\mathcal M_0(T)$ a space of measures on $T=T(F)$ (with properties to be specified). For any $s\in \CC$, consider the distribution 
\begin{equation} \label{DS} D_s:=|x|^s \psi(x) d^\times x
\end{equation}
on $F^\times$. Any cocharacter $\check\lambda: \Gm\to T$ induces, by push-forward, a distribution $\check\lambda_*D_s$ on $T$. The equivariant Fourier transform $\mathscr F_{\check\lambda,s}$ is defined as the operator of convolution by $\check\lambda_* D_s$, on the given space of measures $\mathcal M_0(T)$. 
If $\mathcal M_0(T) = \mathcal S(T)$, the convolution is convergent. In the general case, we will typically need to regularize it. For example, if $T=F^\times$ and  $\mathcal M_0(T) = \mathcal S(F)$ (considered by restriction as measures on $F^\times\subset F$), with $\check\lambda$ the identity cocharacter, we formally have:
$$ \mathscr F_{\check\lambda,s} f (\xi) = \int_{F^\times} |x|^s \psi(x) f(x^{-1} \xi) d^\times x = |\xi|^{-s} \int_{F^\times} |x|^{s-1} f(x^{-1}) \psi(x\xi) dx,$$
and the integral on the right makes sense as the Fourier transform of the distribution $x\mapsto |x|^{s-1} f(x^{-1})$, whenever the latter has a ``natural'' interpretation as a tempered distribution. ``Fourier transform'', here, maps distributions to distributions, since we have fixed the self-dual measure $dx$. 

In general, I will leave ``natural'' interpretations of such regularizations to the reader to figure out, unless there is something highly non-trivial to be pointed out; the typical example of a distribution that needs to be understood as a tempered distribution is of the form $|x|^s f(x) $, where $f\in \mathcal S(F)$ is a Schwartz measure; it is well-known that this is a measure when $\Re(s)>-1$, and makes sense by analytic continuation as a distribution for all but a countable set of values of $s$. 

Finally, we discuss Mellin transforms on tori and functional equations under the equivariant Fourier transforms. Let $f$ be a measure on a torus $T$. Its Mellin transform is a function on the character group of $T$, typically defined by meromorphic continuation of the integral
$$ \check f(\chi) = \int_T f(t) \chi^{-1}(t).$$

Notice that for $T=\Gm$ and $f= \Phi dx$, where $\Phi$ is a Schwartz function on $\Ga$, we have 
$$\check f(\chi) = Z(\Phi, \chi^{-1}, 1),$$
a Tate zeta integral.
The reason for having $\chi^{-1}$, instead of $\chi$ in  the definition of Mellin transforms is that it is more natural to denote by $f\mapsto \check f(\chi)$ the $\chi$-equivariant quotient of $\mathcal S(T)$, despite the fact that this is contrary to the classical definition of Mellin transform on $\Gm$ as $M\Phi(s) = \int \Phi(x) |x|^s d^\times x$.

We have the following:
\begin{proposition}\label{gamma}
 \begin{equation}\label{FE} \widecheck{(\mathscr F_{ {{{\lambda}}} , s} f)}(\chi) = \gamma(\chi,\check\lambda,1-s,\psi) \check f(\chi).\end{equation}
\end{proposition}

Here, 
$$\gamma(\chi,\check\lambda,1-s,\psi) = \frac{\epsilon(\chi,\check\lambda,1-s,\psi) L(\chi^{-1},\check\lambda,s)}{L(\chi,\check\lambda,1-s)}$$ 
is the gamma-factor of the local functional equation, for the local $L$-function attached to the representation $\check\lambda: \check T\to \Gm$ of the dual torus.

\begin{proof}
The transform is equivariant, and hence clearly satisfies the above formula, for some scalar in place of $\gamma(\chi, \check\lambda,1-s,\psi)$. Moreover, by the definition of $\mathscr F_{\check\lambda,s}$, this scalar has to depend only on the pull-back of $\chi$ via $\check\lambda$. Therefore, we are reduced to the case of $T=\Gm$, with $\check\lambda=$ the identity cocharacter.

In this case, let $f = \Phi d^\times t$. By definition,
 $$ \widecheck{\mathscr F_{\lambda, s} f}(\chi)  = \int \mathscr F_{\check\lambda, s} \Phi (t) \chi^{-1}(t) d^\times t = \int \left(\int \Phi(u^{-1}) |u|^s \psi(ut) d^\times u  \right) \chi^{-1}(t) |t|^s d^\times t. $$
 The inner integral is the (classical) Fourier transform of the function $\varphi:u\mapsto |u|^{s-1} \Phi(u^{-1})$, defined as $\hat\varphi(t) = \int \varphi(u) \psi(ut) du$. The whole expression is, therefore, the zeta integral $Z(\hat\varphi,\chi^{-1},s)$. Recall the functional equation\footnote{This convention on local gamma factors was misstated in \cite[Lemma 2.12]{SaBE1}, where $\psi^{-1}$ instead of $\psi$ was used to define Fourier transforms. As a result, $\psi$ should be replaced by $\psi^{-1}$ in all gamma factors of \cite{SaBE1, SaBE2}.} \cite{Tate-Corvallis}
 $$ \gamma(\chi,s,\psi) Z(\varphi,\chi,s) = Z(\hat\varphi, \chi^{-1},1-s).$$
 
 Thus, 
 $$Z(\hat\varphi,\chi^{-1},s) = \gamma(\chi,1-s,\psi) Z(\varphi, \chi, 1-s) = $$
 $$=\gamma(\chi,1-s,\psi) \int |u|^{s-1} \Phi(u^{-1}) \chi(u) |u|^{1-s} d^\times u = \gamma(\chi, 1-s,\psi) \check f(\chi).$$
 \end{proof}
 
\section{Non-standard test measures, and trivialization of Kuznetsov orbital integrals}\label{sec:Kuznetsov}

\subsection{Trivialization of push-forwards for the Kuznetsov formula} \label{sstrivialization}

Let $G=\SL_n$, $\GL_n$ or $\PGL_n$, and let $N$ be the subgroup of upper triangular unipotent matrices. We will consider $\psi$ as a character of $N$, by composing with the sum of the $(i, i+1)$-entries of the matrix ($1\le i\le n-1$). For $\SL_n$, there is also a choice of generic character which is non-conjugate to this, but since our comparisions (for example, with the trace formula in the case of $n=2$) will only involve stable orbital integrals, this would not change anything in the formulas. 

Let $Y\subset G$ be the variety of anti-diagonal elements. The action map $N\times N\times Y\to G$ is an isomorphism with the open Bruhat cell, which we temporarily denote by $Z$. 
A function on $Z$ satisfying $\Phi(n_1 g n_2) = \psi(n_1 n_2) \Phi(g)$ can be thought as a section of a line bundle $\mathcal L_{\psi\times\psi}$ over the points of the variety $N\backslash Z/N$. Let $f$ denote a measure on $N\backslash Z/N$ that is valued on $\mathcal L_{\psi\times\psi}$. The section $Y\xrightarrow\sim N\backslash Z/N$ allows us to trivialize this line bundle, and \emph{consider $f$ as a measure on $Y$}. Explicitly, in the case where these measures are absolutely continuous with respect to a ``Haar'' (=invariant under the action of the diagonal torus) measure on $Y$, we have $f = \Phi dy$, for a function $\Phi$ as above, and the corresponding scalar-valued measure $f_Y$ on $Y$ is:
\begin{equation}
 f_Y (y) = \Phi(y) dy.
\end{equation}
In the rest of this paper, we will not be distinguishing, notationally, between $f$ and $f_Y$, that is: we will always be considering such a measure $f$ as a scalar-valued measure on the variety $Y$ of anti-diagonal elements (which we will also be identifying with the quotient $N\backslash Z/N$).

Let, now, $\mathcal L_\psi$ denote the line bundle defined (similarly) over $N\backslash G$ by the character $\psi$ with respect to the left multiplication/right action of $N$ on functions; explicitly, its sections are functions $\Phi$ on $G$ which satisfy $\Phi(ng) = \psi(n) \Phi(g)$. Let $\mathcal S(N,\psi\backslash G)$ denote the space of Schwartz measures valued in this line bundle. There is a canonical ``twisted push-forward'' map
\begin{equation}\label{pushforward-Kuz} p_!: \mathcal S(N,\psi\backslash G) \hat\otimes \mathcal S(N,\psi^{-1}\backslash G) \to \mathcal M(N\backslash Z/N, \mathcal L_{\psi\times\psi}) \xrightarrow{\sim} \mathcal M(Y)
\end{equation}
(measures on $Y$), with the last isomorphism provided by the comments above. This map is dual to the pull-back of sections of the dual line bundle:
$$ \Gamma(N\backslash Z/N, \mathcal L_{\psi^{-1}\times\psi^{-1}}) \xrightarrow{p^*} \Gamma(N^2\backslash \pi^{-1} Z, \mathcal L_{\psi^{-1}} \boxtimes \mathcal L_\psi),$$
where $p: G \times G \to G$ is the map $(g_1, g_2) \mapsto g_1 g_2^{-1}$ (descending to a map $N\backslash G\times N\backslash G \to N\backslash G/N$). We can also pull-back elements of $\Gamma(N\backslash Z/N, \mathcal L_{\psi^{-1}\times\psi^{-1}})$ to $(N,\psi^{-1})$-equivariant sections of $\mathcal L_{\psi^{-1}}$ over $N\backslash Z$, and thus think of the twisted push-forward as a map 
\begin{equation}
 p_!: \mathcal S(N,\psi\backslash G)\to \mathcal M(N\backslash Z/N, \mathcal L_{\psi\times\psi}) \xrightarrow{\sim} \mathcal M(Y).
\end{equation}

Moreover, identifying $Y$ with $N\backslash Z/N$, consider the embedding
$$ Y\hookrightarrow \mathfrak C:= N\backslash G \sslash N,$$
so that we can identify measures on $Y$ as measures on $\mathfrak C$ (simply by letting the mass of the complement be zero).
It will also be convenient to identify $\mathfrak C$ as an embedding of $A = B/N$, the universal Cartan of $G$, where $B$ is the Borel subgroup of upper triangular matrices. To perform this identification, we identify $A$ with $Y$ via the map
\begin{equation}
 a\mapsto wa,
\end{equation}
where $w$ is a representative for the longest element of the Weyl group. (The precise choice of $w$ is not very important, and will only be fixed when necessary.) Under this identification, $\mathfrak C$ is the toric embedding of $A$ associated to the cone of positive coroots, i.e., a limit $\lim_{t\to 0} \check\lambda(t)$ exists in $\mathfrak C$ if and only if $\check\lambda$ is a cocharacter into $A$ in the cone of positive coroots.

\begin{lemma}
The map \eqref{pushforward-Kuz} has image in the space of finite measures on $\mathfrak C$.
\end{lemma}

\begin{proof}
 Let $f\in  \mathcal S(N,\psi\backslash G) \hat\otimes \mathcal S(N,\psi^{-1}\backslash G)$. The absolute value $|f|$ is a positive Schwartz measure on $N\backslash G \times N\backslash G$, and the push-forward of $f$, considered as a measure on $\mathfrak C$, is bounded by the push-forward of $|f|$.
\end{proof}

We will denote the image of \eqref{pushforward-Kuz} by $\mathcal S(N,\psi\backslash G/N,\psi)$, considered as a space of measures on $\mathfrak C$.

Finally, we discuss the relation between \emph{functions, measures, and half-densities}: Fixing a Haar measure on $N$, the orbital integrals of a Schwartz function $\Phi\in \mathscr F(G)$:
$$ O_t (\Phi) = \int_{N\times N} \Phi(n_1 w t n_2) \psi^{-1}(n_1 n_2) dn_1 dn_2$$
define a function on $A$ (identified as above with the anti-diagonal $Y$). On the other hand, fixing also a Haar measure on $G$, the push-forward $f$ of $\Phi dg$ to $\mathcal S(N,\psi\backslash G/N,\psi)$ is a measure on $A\simeq Y\subset\mathfrak C$. It is easy to compute the integration formula:
\begin{equation}\label{integration}
 f(t) = \delta(t) O_t (\Phi) dt,
\end{equation}
for some Haar measure $dt$ on $A$. 

In some cases, it will be easier to work with \emph{half-densities}, instead of measures. These are sections of the complex line bundle that is obtained from the positive square root of the $\RR^\times$-torsor that gives rise to the bundle of real-valued measures. We will denote Schwartz half-densities by $\mathcal D$. For example, a Schwartz half-density on a group $G$ is equal to $\Phi (dg)^\frac{1}{2}$, where $(dg)^\frac{1}{2}$ is a Haar half-density. 

As with functions, there is no natural push-forward of half-densities for an arbitrary map $X\to \mathfrak C$ of spaces. However, here that our space is a quotient by a group action which is generically free, it is natural to define the push-forward map
$$p_!: \mathcal D(G) \to \mathcal D(N,\psi\backslash G/N,\psi)$$
(where $\mathcal D(N,\psi\backslash G/N,\psi)$ is, by definition, a space of half-densities on $A\subset \mathfrak C$) by integrating over the fibers against our fixed Haar half-density on $N\times N$. Based on \eqref{integration}, it will be given 
by the formula
\begin{equation}\label{pushf-densities}
 p_! (\Phi (dg)^\frac{1}{2}) = \delta^\frac{1}{2}(t) O_t (\Phi) (dt)^\frac{1}{2}.
\end{equation}

As with orbital integrals, half-densities are only defined densely on $N\backslash G\sslash N$, namely on the open subset $A\simeq Y$.

\subsection{Non-standard test measures and densities} \label{ssnonstandard}

We will need to extend the space $\mathcal S(N,\psi\backslash G)$ of Schwartz Whittaker measures, in order to accommodate generating series for local $L$-functions. Actually, the space we are interested in extending is the space $\mathcal S(N,\psi\backslash G/N,\psi)$ of Kuznetsov orbital integrals, but sometimes this extension will just come from an extension of $\mathcal S(N,\psi\backslash G)$. Here are the cases that we will be concerned with in this paper:

\subsubsection{Test measures for the standard $L$-function of $\GL_n$.}

Let $G=\GL_n$, and consider the space $\mathcal S(\Mat_n)$ of Schwartz measures on $n\times n$-matrices, for some $s\in \CC$. Consider the twisted push-forward map:
\begin{equation}\label{fromMat} \mathcal S(\Mat_n) \to \mathcal M(N,\psi\backslash G),\end{equation}
composed with the twisted push-forward to $\mathfrak C = N\backslash G\sslash N$, which was defined via a trivialization in the previous subsection:
$$ p_!: \mathcal M(N,\psi\backslash G) \to \mathcal M(\mathfrak C).$$
The image of the composition of the two maps will be denoted by 
$$\mathcal S^-_{L(\Std, \frac{n+1}{2})}(N,\psi\backslash G/N,\psi).$$
 The measures in $\mathcal S^-_{L(\Std, \frac{n+1}{2})}(N,\psi\backslash G/N,\psi)$ are finite, because they are bounded by the corresponding push-forward measures to $N\backslash G\sslash N$ without the character. The reason for the index $L(\Std, \frac{n+1}{2})$ is the following: 
 
 \begin{proposition}\label{basic-std}
 Let $F$ be non-Archimedean, with ring of integers $\mathfrak o$. Let $\varphi_0$ be the Whittaker measure which is the product of a Haar (=invariant) measure on $N\backslash G$ by the Whittaker function which is equal to $1$ on $K$ and zero off $NK$. The image of \eqref{fromMat} contains a $K=\GL_n(\mathfrak o)$-invariant Whittaker measure $\varphi$, which admits the following description:
 
 For every $i\ge 0$, the restriction of $\varphi$ to $\{g\in \GL_n | \val(\det(g)) = i\}$ is equal to 
 $$q^{-i\cdot \frac{n+1}{2}} h_{\Sym^i}\star \varphi_0,$$ 
 where $h_{\Sym^i}$ is the element of the unramified Hecke algebra that corresponds, under the Satake isomorphism, to the trace of the $i$-th symmetric power of the standard representation, and $\star$ denotes the convolution operator
 $$ h\star \varphi(x) = \int_G h(g) \varphi(xg^{-1}) dg.$$
 For $i<0$, the restriction of $\varphi$ to the corresponding set is zero.
 \end{proposition}
 
 \begin{remark}
  The Satake transform $h\mapsto \check h \in \CC[\check G]^{\check G}$  (where $\check G$ denotes the dual group) is normalized so that for $v_{K,\chi} = $ the unramified vector in a principal series representation $\pi_\chi$ obtained by normalized induction from an unramified character $\chi$ of the Borel subgroup, we have $\pi_{\chi}(h) v_{K,\chi} = \check h(\check\chi) v_{K,\chi}$, where $\check\chi$ is the element of the dual Cartan $\check A\subset\check G$ with $\chi({\check\lambda}(\varpi)) = {\check\lambda}(\check\chi)$ for any coweight $\check\lambda$ into $A$ and $\varpi$ is a uniformizer in $\mathfrak o$. 
 \end{remark}
 
 \begin{remark}
  A small detail here: the right regular action of an element $h$ in the Hecke algebra on Whittaker functions or measures is defined as 
  $$h \cdot \varphi(x) = \int_G h(g) \varphi(xg) dg,$$
  so it is related to the convolution operator defined above by $h\cdot \varphi = h^\vee \star \varphi$, where $h^\vee(g)=h(g^{-1})$.
 \end{remark}

 \begin{proof} The measure
  $\varphi$ is simply the image of $\Phi=$ the characteristic function of $\Mat_n(\mathfrak o)$ times an additive Haar measure (which is equal to $|\det|^n$ times a multiplicative Haar measure), and $\varphi_0$ is the image of $\Phi_0=$ the characteristic function of $\GL_n(\mathfrak o)$ times a multiplicative Haar measure. We may fix the Haar measure $dg$ on $G$ with $dg(G(\mathfrak o))=1$. The result now follows from Godement--Jacquet theory. More precisely, the statement of the proposition is equivalent to the statement that 
  $$ \Phi = \sum_{i\ge 0} q^{i\cdot \frac{n-1}{2}} h_{\Sym^i}\star \Phi_0,$$
  which in turn is equivalent to the statement that if 
  $$M_\chi (g) = \left< \pi_\chi(g) v_{K,\chi}, v_{K,\chi^{-1}}\right>$$ 
  is the unramified matrix coefficient for the principal series representation $\pi_\chi$ as in the remark above, normalized so that $M_\chi(1) = 1$, and assuming that the central character of $\pi_\chi$ is large enough so that the integral below converges, we have:
  $$ \int_{\GL_n} \Phi(g) \cdot M_\chi(g) dg = \sum_{i\ge 0} q^{i\cdot \frac{n-1}{2}} \int_{\GL_n} \Phi_0(g)  \cdot \left< \pi_\chi(g) \pi_\chi(h_{\Sym^i}) v_{K,\chi},v_{K,\chi^{-1}} \right> dg =$$ $$= \sum_{i\ge 0} q^{i\cdot \frac{n-1}{2}} \tr_{\Sym^i} (\check\chi) = \frac{1}{\det(I- q^{i\cdot \frac{n-1}{2}}\Std(\check\chi))} = L(\pi_\chi,\Std, -\frac{n-1}{2}).$$
  This is proven in  \cite{GJ}.
 \end{proof}
 
The image of $\varphi$ in $\mathcal S^-_{L(\Std, \frac{n+1}{2})}(N,\psi\backslash G/N,\psi)$ will be called the \emph{basic vector} of this space, and denoted by $f^0_{L(\Std, \frac{n+1}{2})}$. (It is determined up to a choice of Haar measure, which we do not fix in this paper; more precise versions of the theorems on basic vectors, fixing this constant, appear in \cite{SaTransfer1, SaTransfer2}.)

 One can of course multiply everything by $|\det|^s$ and construct a space $\mathcal S^-_{L(\Std, s +\frac{n+1}{2})}(N,\psi\backslash G/N,\psi)$, with analogous properties as in the proposition above. 

 When working with Fourier transforms on $\Mat_n$, it is more convenient to work with half-densities instead of measures. Notice that a Schwartz half-density on $\Mat_n$ is of the form $|\det g|^\frac{n}{2}\Phi(g) (dg)^\frac{1}{2}$, where $(dg)^\frac{1}{2}$ is a Haar half-density on $G$ and $\Phi$ is a Schwartz function. Its pushforward to $A\subset \mathfrak C$ under the twisted push-forward map $p_!$ map defined in \eqref{pushf-densities} will be:
 \begin{equation}\label{pushf-Matn}
  |\det t|^\frac{n}{2} \delta^\frac{1}{2}(t) (dt)^\frac{1}{2} \cdot O_t(\Phi).
 \end{equation}
 Explicitly, if $t = \diag(t_1, \dots, t_n)$, this is 
 $$ |t_1|^{n-\frac{1}{2}} |t_2|^{n-\frac{3}{2}} \cdots |t_n|^\frac{1}{2} \int_{N\times N} \Phi(n_1 w t n_2) \psi^{-1}(n_1 n_2) dn_1 dn_2.$$
 This will be useful for us in \S \ref{sec:standard}. Of course, the image of $\mathcal D(\Mat_n)$ under this twisted push-forward map will be denoted by 
 $$ \mathcal D^-_{L(\Std, \frac{1}{2})}(N,\psi\backslash G/N,\psi).$$

 Finally, consider the dual space $\Mat_n^\vee$. It can be identified with $\Mat_n$ via the trace pairing, but to make this compatible with the $G\times G$-action we need to let $(g_1,g_2)\in G$ act on $X\in \Mat_n^\vee$ (represented by a matrix) via $X\cdot (g_1,g_2) = g_2^{-1} X g_1$. The map $g\mapsto g^{-1} \in \GL_n \subset \Mat_n^\vee$, then, becomes an equivariant embedding.
 
 Thus, we have a similar push-forward map  
 $$|\det|^{-s}\mathcal S(\Mat_n^\vee) \to \mathcal M(\C).$$
 Here $\det$ is the determinant on $\GL_n$, when embedded equivariantly into $\Mat_n^\vee$; under the identification of $\Mat_n^\vee$ with $\Mat_n$, this is $\det^{-1}$ --- hence the negative power in the determinant.
 We will denote the image of this map by 
 $$\mathcal S^-_{L(\Std^\vee, s+\frac{n+1}{2})}(N,\psi\backslash G/N,\psi).$$
 
 Proposition \ref{basic-std} holds for this space, if we replace the standard representation by its dual.

 \subsubsection{Test measures for $\SL_2$ or $\PGL_2$.}
 
 Now let $G=\SL_2$ or $\PGL_2$. We will only discuss non-standard spaces of test measures associated to products of the adjoint $L$-function, when $G=\SL_2$ and the standard $L$-function, when $G=\PGL_2$.
 
 Let $r = \bigoplus_i r_i$ be an irreducible algebraic representation of  the (connected) dual group $\check G$, decomposed in terms of its irreducible constituents, which we assume to be all isomorphic to the adjoint representation $\Ad=\Sym^2\otimes\det^{-1}$ if $G=\SL_2$, i.e., the irreducible represenation with highest weight $\check\alpha$, and isomorphic to the standard representation (highest weight $\frac{\check\alpha}{2}$) if $G=\PGL_2$. Let $\underline s =(s_i)_i$ be a collection of complex numbers, one for each $r_i$, we define a space of test measures 
 $$ \mathcal S^-_{L(r,\underline s)} (N,\psi\backslash G/N,\psi) = \mathcal S^-_{\prod_i L(r_i,s_i)} (N,\psi\backslash G/N,\psi),$$
 containing $\mathcal S(N,\psi\backslash G/N,\psi)$, as follows:
 
 Recall from \S \ref{sstrivialization} that $\mathcal S(N,\psi\backslash G/N,\psi)$ has been  defined as a space of measures on $\mathfrak C = N\backslash G\sslash N$. We identify $\mathfrak C$ with $\Ga$, in a way compatible with the sections
$$ \tilde\zeta = \begin{pmatrix} & -\zeta^{-1} \\ \zeta \end{pmatrix} \,  \, \, (\mbox{for }\SL_2),$$
 $$\tilde \xi = \begin{pmatrix} & -1 \\ \xi \end{pmatrix} \, \,\, (\mbox{for }\PGL_2)$$
 over $\Gm$. 

 We let $ \mathcal S^-_{L(r,\underline s)} (N,\psi\backslash G/N,\psi)$ be the space of measures on $\mathfrak C = \Ga$ which on any compact set coincide with elements of $\mathcal S(N,\psi\backslash G/N,\psi)$, while in a neighborhood of infinity they are of the form  
 \begin{equation}\label{expansionSL2}
  \sum_i C_i(\zeta^{-1}) |\zeta|^{1-s_i} d^\times \zeta,
 \end{equation}
in the case of $G=\SL_2$, and
 \begin{equation}\label{expansionPGL2}
  \sum_i C_i(\xi^{-1}) |\xi|^{\frac{1}{2}-s_i} d^\times \xi,
 \end{equation}
in the case of $G=\PGL_2$, where the $C_i$'s are smooth functions in a neighborhood of zero, \emph{except} when two or more of the $s_i$'s coincide, in which case the corresponding contributions will be $(C_1(\zeta^{-1}) + C_2(\zeta^{-1})\log|\zeta| + C_3(\zeta^{-1}) \log^2|\zeta| +\dots) |\zeta|^{1-s} d^\times \zeta$ (as many summands as occurences of the exponent $s$). 

\begin{remark}
 Under our previous identification of $\mathfrak C$ with an embedding of the universal Cartan $A$, the factors $|\zeta|^1$ and $|\xi|^\frac{1}{2}$ read $\delta^{\frac{1}{2}}(t)$, and are volume factors. On the other hand, the factors $|\zeta|^{-s_i}$ and $|\xi|^{-s_i}$ are equal to $q^{s_i}$ on an element of the form $e^{\check\alpha}(\varpi)$, resp.\ $e^\frac{\check\alpha}{2}(\varpi)$, when $F$ is a non-Archimedean field with residual degree $q$ and uniformizer $\varpi$; these factors are related to the Taylor series expansion of the local $L$-function.  
\end{remark}

The reason for this definition is the following, which can be proven as in \cite[Lemma 5.3]{SaBE1}:

 \begin{proposition}\label{basic-slpgl}
 Let $F$ be non-Archimedean, with ring of integers $\mathfrak o$, and set $K=G(\mathfrak o)$. Let $\varphi_0\in \mathcal S(N,\psi\backslash G)$ be the product of a Haar (=invariant) measure on $N\backslash G$ by the Whittaker function which is equal to $1$ on $K$ and zero off $NK$. Let $$\varphi = \prod_i\left(\sum_j q^{-js_i} h_{\Sym^j r_i}\right)\star \varphi_0,$$
 where $h_{\Sym^j r_i}$ is the element of the unramified Hecke  algebra with Satake transform equal to the trace of the $j$-th symmetric power of $r_i$, and the product over $i$ is the convolution product. This series is locally finite as a series of measures, and hence makes sense as a measure on $N\backslash G$ valued in the line bundle defined by the character $\psi$.
 
 Then, the twisted push-forward of $\varphi$ to $\mathfrak C$ belongs to $\mathcal S^-_{L(r,\underline s)} (N,\psi\backslash G/N,\psi)$. 
 
 \end{proposition}

The image of $\varphi$ in $\mathcal S^-_{L(r,\underline{s})}(N,\psi\backslash G/N,\psi)$ will be called the \emph{basic vector} of this space, and denoted by $f^0_{L(r,\underline s)} $. (It is determined up to a choice of Haar measure, which can be fixed globally.)

We may again work with half-densities, in which case the analogous to \eqref{expansionSL2}, \eqref{expansionPGL2} expansions are:
\begin{equation}\label{expansionSL2-densities}
  \sum_i C_i(\zeta^{-1}) |\zeta|^{-s_i} (d^\times \zeta)^\frac{1}{2},
 \end{equation}
in the case of $G=\SL_2$, and
 \begin{equation}\label{expansionPGL2-densities}
  \sum_i C_i(\xi^{-1}) |\xi|^{-s_i} (d^\times \xi)^\frac{1}{2},
 \end{equation}
by the integration formula \eqref{integration}.

The space $\mathcal S^-_{L(\Std,s+\frac{3}{2})}(N,\psi\backslash \PGL_2/N,\psi)$ can be obtained from the space 
 $\mathcal S^-_{L(\Std,s+\frac{3}{2})}(N,\psi\backslash \GL_2/N,\psi)$ defined in the previous subsection by integrating over the center. More precisely:
 
\begin{proposition}
 Consider the  space of measures $\mathcal S^-_{L(\Std,s+\frac{3}{2})}(N,\psi\backslash \GL_2/N,\psi)$ (obtained by twisted push-forward from the space $|\det|^s \mathcal S(\Mat_2)$), and the push-forward
 $$\mathcal S^-_{L(\Std,s+\frac{3}{2})}(N,\psi\backslash \GL_2/N,\psi) \to  \mathcal M(N\backslash \PGL_2\sslash N)$$
 under the natural map $N\backslash G\sslash N\to N\backslash \PGL_2\sslash N$.
 
 This push-forward map converges (has image in locally finite measures) for $\Re(s)\gg 0$, has image equal to $\mathcal S^-_{L(\Std,s+\frac{3}{2})}(N,\psi\backslash \PGL_2/N,\psi)$, and admits holomorphic continuation as a map 
 $$ \mathcal S^-_{L(\Std,s+\frac{3}{2})}(N,\psi\backslash \GL_2/N,\psi) \to \mathcal S^-_{L(\Std,s+\frac{3}{2})}(N,\psi\backslash \PGL_2/N,\psi)$$
 for all $s\in \CC$, in the sense that the image of $|\det|^s f$ is a holomorphic section\footnote{For the topological properties of extended Schwartz spaces and the notion of holomorphic sections, I point the reader to \cite[Appendix A]{SaBE2}.} of $\mathcal S^-_{L(\Std,s+\frac{3}{2})}(N,\psi\backslash \PGL_2/N,\psi)$, for any $f\in \mathcal S^-_{L(\Std,\frac{3}{2})}(N,\psi\backslash \GL_2/N,\psi)$. 
 
 Finally, this map is surjective, and maps the basic vector of $\mathcal S^-_{L(\Std,s+\frac{3}{2})}(N,\psi\backslash \GL_2/N,\psi)$ to the basic vector of $\mathcal S^-_{L(\Std,s+\frac{3}{2})}(N,\psi\backslash \PGL_2/N,\psi)$.
\end{proposition}

\begin{proof}
Since the elements of $\mathcal S^-_{L(\Std,s+\frac{3}{2})}(N,\psi\backslash \GL_2/N,\psi)$ are obtained by twisted push-forward from the space $|\det|^s \mathcal S(\Mat_2)$, we will argue directly in terms of the twisted push-forward map:
$$ |\det|^s \mathcal S(\Mat_2) \to \mathcal M(N\backslash \PGL_2\sslash N),$$
that we will denote by $\overline{p_!}$.
Recall that ``twisted'' refers to the character $\psi$ on $N$.

Notice that $N\bbslash \Mat_2 \sslash N\simeq \mathbbm A^2$ with coordinates $(c,\det)$, where $c$ is the bottom-left coordinate of the matrix $\begin{pmatrix} a & b \\ c & d\end{pmatrix}$. 
First, let $\Mat_2^\circ$ denote the open subvariety with $c\ne 0$, and consider the space $\mathcal S(\Mat_2^\circ)$. The push-forward map $\Mat_2^\circ \to \Gm\times \mathbbm A^1 \subset \mathbbm A^2$ is smooth (submersive) and surjective, and the action of $N\times N$ on $\Mat_2^\circ$ is free, which implies that the twisted push-forward 
$$ \mathcal S(\Mat_2^\circ) \to \mathcal M(\Gm\times\mathbbm A^1)$$
is a surjection onto $\mathcal S(\Gm\times\mathbbm A^1)$. Dividing by the action of the center we get a submersive surjection 
$$ \Gm\times \mathbbm A^1 \ni (c,\det)\mapsto  \frac{\det}{c^2}\in \mathbbm A^1,$$
and hence a push-forward map 
$$ \mathcal S(\Gm\times \mathbbm A^1) \twoheadrightarrow \mathcal S(\mathbbm A^1),$$
and, more generally,
$$ |\det|^s \mathcal S(\Gm\times \mathbbm A^1) \twoheadrightarrow |\xi|^{-s} \mathcal S(\mathbbm A^1),$$
where $\xi^{-1}$ is the coordinate on $\mathbbm A^1$. This matches our coordinate $\xi$ in $N\backslash \PGL_2\sslash N$, and hence the image of the last map consists precisely of the elements of  $\mathcal S^-_{L(\Std,s+\frac{3}{2})}(N,\psi\backslash \PGL_2/N,\psi)$ whose germ at $\xi=0$ vanishes. 

On the other hand, the subspace $|\det|^s \mathcal S(\GL_2) = \mathcal S(\GL_2)$ of $|\det|^s \mathcal S(\Mat_2) $ has twisted push-forward equal, by definition, to $\mathcal S(N,\psi\backslash \PGL_2/N,\psi)$. Thus, we get a well-defined surjection
$$ \overline{p_!}: |\det|^s \mathcal S(\GL_2 \cup \Mat_2^\circ) \twoheadrightarrow \mathcal S^-_{L(\Std,s+\frac{3}{2})}(N,\psi\backslash \PGL_2/N,\psi),$$
for every $s$, which clearly is holomorphic in $s$.

To extend this map to $|\det|^s \mathcal S(\Mat_2)$, first of all we notice that the twisted push-forward map $\overline{p_!}$ indeed converges for $\Re(s)\gg 0$ --- in fact, already for $s=0$ because it is bounded by the untwisted push-forward map from finite measures on $\GL_2$ (by restriction) to finite measures on $N\backslash \PGL_2\sslash N$. Notice also that $p_!$ is $(N,\psi)^2$-equivariant. 

Now, consider the short exact sequence
$$ 0 \to |\det|^s \mathcal S(\GL_2 \cup \Mat_2^\circ) \to |\det|^s \mathcal S(\Mat_2) \to |\det|^s \mathcal S_{\{c=0, \det=0\}}(\Mat_2) \to 0,$$
where the space on the right is, by definition, the \emph{stalk} of $\mathcal S(\Mat_2)$ at the closed subset $\{c=0, \det=0\}$.

We claim that the $(N,\psi)^2$-coinvariants of the space $|\det|^s \mathcal S_{\{c=0, \det=0\}}(\Mat_2)$ (equivalently, of the space $\mathcal S_{\{c=0, \det=0\}}(\Mat_2)$) are zero; equivalently, there are no $(N,\psi)^2$-equivariant distributions supported on the set  $Y=\{c=0, \det=0\}$. To see that, stratify this set in terms of the entries $(a,d)$ on the diagonal of a matrix: we have open strata $\{a=0, d\ne 0\}$ and $\{a\ne 0, d=0\}$, and the closed stratum $\{a=d=0\}$. If $Y_1$ is one of these strata, we will prove that there are no $(N,\psi)^2$-equivariant distributions on an open neighborhood of $Y_1$ and supported on $Y_1$, and by decreasing induction on the dimension of the strata this will imply that there are no $(N,\psi)^2$-equivariant distributions on $\Mat_2$ supported on $Y$.

Let $Y_1$ be one of these strata. The subgroup $N_0= $ $N\times 1$ or $1\times N$ of $N^2$ (depending on $Y_1$) stabilizes every point on $Y_1$. Moreover, $Y_1$ has a semi-algebraic open neighborhood which is $N_0$-equivariantly isomorphic to $Y_1 \times V$, where $V$ is the fiber of the normal bundle to $Y_1$ at a point $y\in Y_1$. Thus, it is enough to show that there are no $(N_0,\psi)$-equivariant distributions on the space $V$ supported on $y$. But, such distributions are necessarily evaluations of derivatives on $y$, and the action of $N_0$ on the space of all derivatives is locally finite, with only the trivial generalized eigencharacter. Hence, there are no $(N_0,\psi)$-equivariant distributions on $V$ supported on $y$, and hence there are no $(N_0,\psi)$-equivariant distributions in a neighborhood of $Y_1$, supported on $Y_1$. 

This means that the map of coinvariants
$$ |\det|^s \mathcal S(\GL_2 \cup \Mat_2^\circ)_{(N,\psi)^2} \to |\det|^s \mathcal S(\Mat_2)_{(N,\psi)^2}$$
is surjective. In the non-Archimedean case, where the functor of $(N,\psi)^2$ coinvariants is exact, this immediately implies that this map is an isomorphism, allowing for a canonical extension of $\overline{p_!}$ to all of $|\det|^s \mathcal S(\Mat_2)$ (which coincides with the twisted push-forward in the convergent region). In the Archimedean case, it is possible to argue again that this map is an isomorphism, but it is easier to observe that, if there was a kernel $|\det|^s\cdot K$ to this map, in the region of convergence the push-forward map $\overline{p_!}$ is zero on this kernel. By holomorphicity, it has to be zero for all $s$, and thus the push-forward extends to $|\det|^s \mathcal S(\Mat_2)$ (and has the same image).

Finally, the statement on basic vectors follows immediately from their characterization in terms of the Satake isomorphism in Propositions \ref{basic-std}, \ref{basic-slpgl}.
\end{proof}

\section{Lifting from the universal Cartan} \label{sec:Cartan} 

Let $H$ be a reductive group. The reductive (torus) quotient of any Borel subgroup of $H$ is called the universal Cartan $A_H$; it is unique up to unique isomorphism induced by any element conjugating one Borel subgroup to another, and it is defined over the base field even if $H$ does not have a Borel defined over it (i.e., is not quasi-split). The $L$-group ${^LA_H}$ embeds canonically (by construction) into the $L$-group $\LH$. 

Now assume that $H$ is quasi-split. This map of $L$-groups corresponds to the following map of representations, at least generically: if $\mathcal W_F$ denotes the Weil group of $F$, a morphism
$$ \varphi: \mathcal W_F \to {^LA_H} \subset \LH$$
corresponds to a character $\chi$ of $A_H$, and to the principal series representation $\pi = I(\chi)= \Ind_B^H(\chi\delta^\frac{1}{2})$ of $H$ (where $\delta$ is the modular character of the Borel subgroup $B$). 

The transfer map $\mathcal T: \mathcal S(H) \to \mathcal S(A_H)$ should be such that the pull-back of a character of $A_H$ (viewed as a functional on $\mathcal S(A_H)$) is equal to characters of the corresponding principal series representations. In fact, it is not completely clear at first that this should be so: notice that for non-unitary characters, the corresponding principal series can have irreducible subquotients belonging to different $L$-packets --- so, one could ask that the pull-back of a character of $A_H$ should be just the character of the subquotient with the associated Langlands parameter. One easily sees, however, that this is too much to ask, and does not lead to a map that preserves ``nice'' spaces of test functions. Therefore, the prevailing philosophy of ``beyond endoscopy'' only requires the transfer to be compatible with functoriality for \emph{tempered} characters. (More generally, one could ask about lifting from tempered representations of one group to unitary Arthur representations of another with a  fixed ``Arthur-$\SL_2$'' parameters.)  In our case, this translates to the requirement that the pull-back of a \emph{unitary} character of $A_H$ via the transfer map should be the character of the corresponding --- \emph{tempered} --- principal series representation. It turns out that there is a unique map which has this property. Here is the construction:

The characters of principal series representations are well-known (s.\ \cite{vanDijk} for a more general formula for induced representations): On regular semisimple elements belonging to a maximally split torus $T$, identified with $A_H$ through its inclusion into some Borel subgroup, the character of $I(\chi)$ is 
$$\Theta_\chi (t)=  D_H^{-\frac{1}{2}}(t) \sum_{w\in W} {^w\chi}(t),$$
where $W$ denotes the relative Weyl group of $H$, and $D_H$ is the discriminant of the Weyl integration formula:
$$ \int_H \Phi(h) dh = \frac{1}{|W|}\int_T \int_{T\backslash H} \Phi(g^{-1} t g) dg D_H(t) dt$$
(when $\Phi$ is supported in the closure of the union of maximally split tori).
The character is zero off the closure of the union of maximally split tori.

Thus, our requirement for the transfer operator: 
$$ \int_{A_H} \mathcal T f (t) \chi(t)  = \int_H f(h) \Theta_\chi (h),$$
for a Schwartz measure $f = \Phi dh$ (where $\Phi$ is a Schwartz function on $H$) reads:

$$\int_H f(h) \Theta_\chi (h) = \frac{1}{|W|} \int_T \int_{T\backslash H} \Phi(g^{-1} t g) dg  D_H^{-\frac{1}{2}}(t) \sum_{w\in W} {^w\chi}(t) D_H(t) dt = $$
$$ = \int_T D_H^\frac{1}{2}(t) \int_{T\backslash H} \Phi(g^{-1} t g) dg  \cdot \chi(t) dt,$$
hence:
$$ \mathcal T(f) =  D_H^\frac{1}{2}(t) \left(\int_{T\backslash H} \Phi(g^{-1} t g) dg\right)  \cdot dt.$$

Choosing a Borel subgroup $B=TN$ (hence identifying $T=A_H$), we can decompose the integral over $T\backslash H$ as an integral over $N$, followed by an integral over $B\backslash H$ (where the integrand varies by the modular character $\delta$ of $B$ on the left), and make the change of variables 
$$ \int_N  \Phi(n^{-1} t n ) dn  = D_H^{-\frac{1}{2}}(t) \delta^\frac{1}{2}(t)  \int_N  \Phi(t n )  dn$$  
to finally arrive at the following formula for the transfer operator:
\begin{equation}\label{transfer-Cartan}
 \mathcal T(f) = \mathcal T(\Phi dh) = \delta^\frac{1}{2}(t) \left(\int_{B\backslash H} \int_N \Phi(g^{-1} t ng ) dn dg \right)\cdot dt. 
\end{equation}

Notice that the image lies in $\mathcal S(A_H)$, and of course it is $W$-invariant. 

Let me briefly make a comment about the global comparison, and how one can prove the lifting of automorphic characters from $A_H$ to $H$ using the trace formula: Let $k$ be a global field, and $\adele$ its ring of adeles. Of course, the stable trace formula for $A_H$, viewed as a functional on $\mathcal S(A_H(\adele))$ 
is just the generalized function
$$ \varphi \mapsto \sum_{\alpha\in A_H(k)} \frac{\varphi}{dt} (\alpha),$$
where $dt$ denotes Tamagawa measure. It is slightly difficult to compare its pull-back to $\mathcal S(H(\adele))$ via the transfer operator $\mathcal T$ with the usual, scalar-valued, (stable) trace formula for $H$. However, it is more natural to think of the trace formula for $H$ as valued in Laurent expansions around a point on the complex plane; this point of view is explained in \cite[\S 4.3]{SaSelberg}. For example, for a test measure $f$ on $H=\PGL_2$, the trace formula will be a formal expression of the form 
$$ \frac{\TF_{-1}(f)}{s} + \TF_0(f),$$
where $\TF_0$ is the usual, \emph{non-invariant} Selberg trace formula, and $\TF_{-1}$ is an invariant distribution which is equal to 
$$ -\int_{B\backslash H(\adele)} K_{f,B} (g, g) dg,$$
where $K_{f,B}$ is the kernel function for the action of $f$ on $L^2(A_H(k)N(\adele)\backslash H(\adele))$:
$$ K_{f,B}(x,y) = \sum_{\alpha\in A_H(k)} \int_{N(\adele)} \frac{f}{dh} (x^{-1} \alpha n y) dn.$$

This that has a spectral expansion  \emph{only  in terms of the Eisenstein spectrum}, and a geometric expansion as a sum over $A_H(k)$:
$$ -\int_{B\backslash H(\adele)} K_{f,B}(g,g) dg= -\sum_{\alpha\in A_H(k)} \int_{B\backslash H(\adele)} \int_{N(\adele)} \frac{f}{dh} (g^{-1} \alpha n g) dn dg.$$

In other words:

\begin{equation} \TF_{-1} (f) = -  \sum_{\alpha\in A_H(k)} \frac{\mathcal T f}{dt} (\alpha).\end{equation}

This is the global comparison between the trace formula for $H=\PGL_2$ and the trace formula for its universal Cartan, via the local transfer operator. (For a general semisimple group of split rank $r$, the Laurent expansion will be of order $r$, and the coefficient of $s^{-r}$ will be the invariant distribution which is comparable to the trace formula of $A_H$.) 

\section{Lifting from the boundary degeneration} \label{sec:degen}

In this section we let $X$ be a homogeneous, spherical, quasi-affine $G$-variety over a \emph{non-Archimedean field $F$ in characteristic zero}. A main theme of the present paper is the comparison between transfer operators involving such a variety $X$, and transfer operators involving its \emph{boundary degeneration}, as horospherical $G$-space which is, roughly, responsible for the continuous spectrum of $X$. 

We will assume that $G$ is split, since this case has been covered with more completeness in the literature, and we also assume that $X$ satisfies the technical assumptions of \cite[\S 2.1]{DHS} --- these ensure that $X$ satisfies the Paley--Wiener theorem \cite[Theorem 1.8]{DHS}, which we will recall below. We will also assume, for convenience, that $X$ carries a $G$-invariant measure. The assumptions are satisfied for all the varieties discussed in the present paper, namely: the group $\SL_2$, the space $\Gm\backslash \PGL_2$ and the ``Whittaker space'', i.e., the space $N\backslash G$ equipped with the line bundle defined by a non-degenerate character of the maximal unipotent subgroup $N$. 

We first recall the notion of a boundary degeneration, or asymptotic cone. The coordinate ring $F[X]$ is a multiplicity-free sum of irreducible $G$-modules, $F[X] = \bigoplus_\lambda F[X]_\lambda$, where $\lambda$ varies in a saturated monoid $\Lambda_X^{++}$ of dominant weights for $G$. The product $F[X]_\lambda\cdot F[X]_\mu$ has image in the sum of spaces $F[X]_\nu$, for some weight $\nu$ satisfying $\nu\preceq \lambda+\mu \overset{\mbox{def}}\iff (\lambda+\mu)-\nu \in R_G$, the positive root cone of $G$. This gives rise to a filtration of $F[X]$ by weights, with respect to the order $\succeq$. The corresponding $\Lambda_X^{++}$-graded ring $\gr F[X]$ is the coordinate ring of a horospherical affine $G$-variety $X_\emptyset^a$, whose open $G$-orbit, the (total) \emph{boundary degeneration} of $X$ we denote by $X_\emptyset$. 

The grading corresponds to an action of $A_X=$ the split torus with character group $\Lambda_X$ (= the group generated by $\Lambda_X^{++}$) on $X_\emptyset^a$ by $G$-automorphisms. The group $A_X$ will be called the \emph{universal Cartan} $A_X$ of $X$; it can also be defined as the quotient by which the universal Cartan $A$ of $G$ acts on the invariant-theoretic quotient $X\sslash N$ (where $N\subset B\twoheadrightarrow A$ is the unipotent radical of a Borel subgroup used to define $A$). The action of $A_X$ on the boundary degeneration $X_\emptyset$ can explicitly be described as follows:  $X_\emptyset$ is isomorphic (up to a choice of base point) to $A_X\times^{P(X)^-} G$, where $P(X)^-$ is opposite to the parabolic subgroup $P(X)\subset G$ stabilizing the open Borel orbit, and the notation $\times^H$ means the quotient of the product by the diagonal action of $H$. For example, in the group case $X=H$ it can be identified with $T\times^{B_H\times B_H^-} (H\times H)$, where $B_H$, $B_H^-$ are two opposite Borel subgroups of $H$, and $T$ is their intersection (a Cartan subgroup of $H$). Clearly, all the usual spaces of functions, such as  $L^2$ and Schwartz functions, on $X_\emptyset$ are parabolically induced from the corresponding spaces of functions on $A_X$ (using normalized parabolic induction from $P(X)^-$ to $G$).

We define the action of $A_X$ on functions by 
\begin{equation}\label{action-normalized-functions}
 (a\cdot \Phi)(x) = \delta_{P(X)}^{\frac{1}{2}}(a) \Phi( a\cdot x),
\end{equation}
on half-densities by 
\begin{equation}\label{action-normalized-densities}
 (a\cdot f)(x) = f( a\cdot x),
\end{equation}
and on measures by
\begin{equation}\label{action-normalized-measures}
 (a\cdot \mu)(x) = \delta_{P(X)}^{-\frac{1}{2}}(a) \mu( a\cdot x),
\end{equation}
where $\delta_{P(X)}$ is the modular character of the parabolic stabilizing the open Borel orbit in $X$. (Explicitly, if we write $X_\emptyset = S \backslash G$, where $S$ belongs to a parabolic $P$ of the class opposite to $P(X)$, so $A_X = P/S$, we have $\delta_{P(X)} = \delta_P^{-1}$; the above normalized actions are isometries on $L^2$-spaces.)

Although we will not explicitly use it, it is helpful to think of the following construction, where $X$ degenerates to $X_\emptyset$: Any cocharacter $\check\lambda:\Gm\to A_X$ which is the image of a strictly dominant cocharacter into $A$ defines a homomorphism $\Lambda_X \to \mathbb Z$, which reduces the filtration of $F[X]$ to an $\mathbb N$-filtration with the same associated graded. Moreover, the standard Artin--Rees construction \cite{Popov} gives rise to a $G\times\Gm$-family 
$$ \mathcal X \to \Ga$$
(compatible with the action of $\Gm$ on the base), with $\mathcal X\times_{\Ga} \Gm = X^a\times\Gm$ (where $X^a=\spec F[X]$), and $\mathcal X_0:=\mathcal X\times_{\Ga}\{0\} = X_\emptyset^a$. (The group $\Gm$ acts on $X_\emptyset^a$ via the chosen cocharacter into $A_X \subset \Aut^G(X_\emptyset^a)$.) There is also a variant of this construction, where the open set is taken to be the quotient of $X^a\times\Gm$ by the intersection of $\check\lambda(\Gm)$ with the $G$-automorphism group of $X$; for example, for $X=\SL_2$ under the $G=\SL_2\times\SL_2$-action, and $\check\lambda=$the positive coroot, this variant gives $\mathcal X=\Mat_2=$ the Vinberg monoid of $\SL_2$, and the map to $\Ga$ is the determinant map.

The above setup extends to cases of ``Whittaker induction'' from affine homogeneous varieties, cf.\ \cite[\S 2.6]{SV}. This includes the Whittaker model of $G$, which is the only case of Whittaker induction we will encounter in this paper, where the variety is $N\backslash G$, but its points are equipped with a complex line bundle $\mathcal L_\psi$ as in \S \ref{sec:Kuznetsov}. We will use the letter $X$ to refer to the data consisting of the variety and the line bundle; this changes the dual group, compared to the dual group of the variety equipped with the trivial line bundle --- in the Whittaker case, we have $\LX = \LG$, while with trivial character on $N$ the $L$-group of $X$ would just be the $L$-group of the universal Cartan of $G$. Moreover, the degeneration $X_\emptyset$ in the Whittaker case is the same variety $N\backslash G$, but equipped with the trivial line bundle. The family $\mathcal X$ of the previous paragraph (for any dominant coweight $\check\lambda$) is the trivial family $N\backslash G \times \Ga$, but the character defining the line bundle over the point $a\in \Ga$ is the character $x\mapsto \psi(\check\lambda(a) x \check\lambda(a)^{-1})$, thus specializing to the trivial character over zero.

In all cases, the $L$-group of $X_\emptyset$ coincides with the $L$-group of $A_X$, thus we have a canonical embedding
\begin{equation}\label{Ldegen}
 \LX_\emptyset\simeq {^LA_X} \hookrightarrow \LX.
\end{equation}

We let $X^\vee$ be equal to $X$ as a variety, but with the $G$-action twisted by a Chevalley involution, then all the above data for $X^\vee$ are ``dual'' to the data for $X$ --- for example, $A_{X^\vee}$ is obtained from $A_X$ by applying the Chevalley involution to the universal Cartan $A$. In the examples considered in this paper, we have $X\simeq X^\vee$ as $G$-varieties, except that in the Whittaker case one has to invert the character $\psi$ for $X^\vee$. In any case, there is no harm in proceeding with the general discussion. It seems that the correct quotient to consider for the relative trace formula, in order to analyze the spectrum of $X$, is $(X\times X^\vee)/G$.

In analogy to the functorial lift discussed in \S \ref{sec:Cartan}, it is natural to ask whether there is a transfer map
$$ \mathcal S(X\times X^\vee/G)\to \mathcal S(X_\emptyset \times X^\vee_\emptyset /G),$$
corresponding to the embedding \eqref{Ldegen} of $L$-groups.
A very natural such map exists, as we shall see, at least in the non-Archimedean case, but we need to enlarge the space of measures on the right. Then the it descends from a canonical ``asymptotics'' map $\mathcal S(X)\to \mathcal S^+(X_\emptyset)$ (and same for $X^\vee$), provided by the theory of asymptotics, where $\mathcal S^+(X_\emptyset)$ is an extension of $\mathcal S(X_\emptyset)$, to be described. 

The asymptotics map is a canonical morphism 
$$ e_\emptyset^*: \mathcal S(X)\to \mathcal S^+(X_\emptyset),$$
characterized by the fact that $\varphi$ and $e_\emptyset^*(\varphi)$ are ``equal close to infinity'', for every $\varphi\in\mathcal S(X)$. For a rigorous definition of what this means, and the construction of the map, cf.\ \cite[\S 5]{SV}.
Notice that there the asymptotics map was defined for functions, but here we consider it as a morphism of measures; there is a canonical way by which a $G$-invariant measure on $X$ induces a $G$-invariant measure on $X_\emptyset$, s.\ \cite[\S 4.2]{SV}, so there is no problem passing from functions to measures.

The space $\mathcal S^+(X_\emptyset)$ was introduced in \cite{DHS}: it is the closure of $\mathcal S(X_\emptyset)$ under a set of ``scattering morphisms'' $\mathfrak S_w: \mathcal S(X_\emptyset) \to \mathcal M(X_\emptyset)$ parametrized by the little Weyl group $W_X$ of $X$. The scattering morphisms give rise to an action of $W_X$ on $\mathcal S^+(X_\emptyset)$ by $G$-automorphisms, which semi-commute with the action of $A_X$, i.e., for $t\in A_X$ we have $$t\cdot \mathfrak S_w f = \mathfrak S_w ({^{w^{-1}}t}\cdot f).$$
(Recall that the action of $A_X$ on measures has been normalized in \eqref{action-normalized-measures}.)
The image of $e_\emptyset^*$ is precisely the subspace of $W_X$-invariants, with $W_X$ acting through the scattering operators --- this the Paley--Wiener theorem \cite[Theorem 1.8]{DHS}.

At this point I stress that the boundary degeneration $X_\emptyset$ should not be considered just as an abstract $G$-variety, but as a $G$-variety equipped with extra structure, which is obtained from $X$. This extra structure is encoded in the scattering morphisms. The space $\mathcal S^+(X_\emptyset)$ is not determined by the abstract variety $X_\emptyset$, but by the variety plus the scattering morphisms, which depend on $X$.

Now consider the morphisms induced by asymptotics:
$$ e_\emptyset^* \otimes e_\emptyset^*: \mathcal S(X)\otimes\mathcal S(X^\vee)\to \mathcal S^+(X_\emptyset)^{W_X} \otimes\mathcal S^+(X_\emptyset^\vee)^{W_X}.$$
One can show that they descend to morphisms on the spaces of push-forward measures: 

\begin{equation}\label{Bcoinv} \mathcal E_\emptyset^*: \mathcal S(X\times X^\vee/G) \to \mathcal S^+(X_\emptyset \times X_\emptyset^\vee/G).\end{equation}

Note that this is not automatic, because $\mathcal S^+(X_\emptyset \times X_\emptyset^\vee/G)$ cannot be identified with the $G^\diag$-coinvariants of $\mathcal S^+(X_\emptyset \times X_\emptyset)$. For example, when $X_\emptyset \simeq X_\emptyset^\vee = N\backslash \SL_2 \simeq \mathbbm A^2\smallsetminus\{0\}$, so $X_\emptyset \times X_\emptyset^\vee/G\simeq N\backslash G/N$, the map $\mathbbm A^2\smallsetminus\{0\}\to \mathfrak C\simeq \mathbbm A^1$ is the projection to a coordinate, but every evaluation on a point in the preimage of $0\in \mathfrak C$ is an $N$-invariant distribution. Nonetheless, using the invariance of the image of the asymptotics morphism $e_\emptyset^*$ under the scattering operators, one can show that the map $\mathcal E_\emptyset^*$ descending from $e_\emptyset^*\otimes e_\emptyset^*$ does exist, uniquely, and one can compute the pull-backs of \emph{relative characters} under this map. We will only present the final result for the examples considered in this paper.

Note that the boundary degeneration $X_\emptyset^\vee$ for $X^\vee$ is isomorphic to $A_X\times^{P(X)} G$, whereas the boundary degeneration for $X$ is $A_X\times^{P(X)^-} G$. Thus, the quotient $\C_\emptyset:= X_\emptyset \times X_\emptyset^\vee \sslash G$ contains an open subset which is an $A_X$-torsor, with the action of $A_X$ descending from the action on $X_\emptyset$. This identification of the open orbit as an $A_X$-torsor is compatible with our identification of an open subset of $N\backslash G\sslash N$ in \S \ref{sstrivialization} with $A$. Indeed, if $w$ denotes a representative of the longest Weyl group element, and $N^-=w^{-1} N w$, then $X_\emptyset$ can be considered as a quotient of $N^-\backslash G$ (non-canonically, up to a choice of base point on an $A$-orbit), $X_\emptyset^\vee$ as a quotient of $N\backslash G$, and hence $X_\emptyset \times X_\emptyset^\vee\sslash G$ as a quotient of $N^-\backslash G\sslash N$. The map $g\mapsto wg$ identifies $N\backslash G\sslash N$ with $N^-\backslash G\sslash N$, and this is the map used in \S \ref{sstrivialization} to identify an open subset of $N\backslash G/N$ with $A$; thus, the $A$-action on $N\backslash G\sslash N$ and on (its quotient) $X_\emptyset \times X^\vee_\emptyset\sslash G$ are compatible. 

Fixing a base point to identify this open $A_X$-torsor with $A_X$, a class of relative characters for $X\times X^\vee$ is given by \emph{Mellin transforms} over this open $A_X$-torsor:
\begin{equation}\label{Mellin-boundary} \check f(\chi) = \int_{A_X} f(t) \chi^{-1}(t).\end{equation} 
For the space $\mathcal S(X_\emptyset\times X_\emptyset^\vee/G)$, this integral will converge for characters in some open region, and have analytic continuation to all characters.
Notice that we do not normalize the action of $A_X$ on measures or functions on $\C_\emptyset$, which is why a shift by $\delta_{P(X)}^{\frac{1}{2}}$ will appear in the formulas that follow. The analog of the integration formula \eqref{integration}, which now requires the modular character $\delta_{P(X)}$ instead of $\delta$, shows that the functional $f\mapsto \check f(\chi\delta_{P(X)}^{\frac{1}{2}})$ is a relative character associated to the normalized induced representation $I_{P(X)}(\chi)$, that is, its pull-back to $\mathcal S(X_\emptyset\times X_\emptyset^\vee)$ factors through a map:
$$\mathcal S(X_\emptyset\times X_\emptyset^\vee) \to I_{P(X)}(\chi)\hat\otimes I_{P(X)}(\chi^{-1}) \to \CC.$$

The pull-back of Mellin transform $f\mapsto \check f(\chi\delta_{P(X)}^{\frac{1}{2}})$ via $\mathcal E_\emptyset^*$ is a relative character $\mathcal I_\chi$ for the same representation, on the space $X\times X^\vee$. In the cases under consideration, where $X\simeq X^\vee$ except for changing the character $\psi$ to $\psi^{-1}$ in the Whittaker case, we have an ``inner product'' on $\mathcal S(X)\otimes \mathcal S(X^\vee)$, and we can ask how this relative character (for $\chi$ unitary) appears in the Plancherel decomposition of the most continuous part (cf. \cite[\S 14.1]{SV}) $L^2(X)_\emptyset$ of $L^2(X)$. The answer is given by the following theorem, which will be proven in \cite{SaTransfer1} by a calculation of scattering operators:

\begin{theorem}\label{thmpullbackfrombd}
There are choices of Haar measures on the spaces under consideration, and a choice of base point on the open $A_X$-orbit on $\C_\emptyset$, compatible with global choices\footnote{That is, when the spaces are defined over a global field $k$, one can choose the base point to be defined over $k$ and the local Haar measures to factorize a Tamagawa measure. This includes the Haar measure on $A_X$ used in the definition of Mellin transform, which induces the Haar--Plancherel measure $d\chi$ on $\widehat{A_X}$ that appears in the statement of the theorem.}, such that the Plancherel decomposition for $L^2(X)_\emptyset$ reads:

$$\int_X \frac{\varphi_1 \cdot \varphi_2}{dx} = \frac{1}{|W_X|}\int_{\widehat{A_X}} \mathcal I_\chi(\varphi_1\otimes\varphi_2) \mu_X(\chi) d\chi,$$
where the scalar $\mu_X(\chi)$ is given by the following formulas:

\begin{itemize}
 \item  For the Whittaker case, $X = N\backslash G$ equipped with a character $\psi\circ $(a sum of non-zero functionals on the simple root spaces),
 \begin{equation} \mu_X(\chi) = \prod_{\alpha>0} \gamma(\chi, \check\alpha, 0, \psi^{-1}).\end{equation}

 \item   For the group case, $X = H$ under the $G=H\times H$-action, 
 \begin{equation} \mu_X(\chi) = \prod_{\alpha>0} \gamma(\chi, -\check\alpha, 0, \psi^{-1}) \cdot \gamma(\chi, \check\alpha, 0, \psi^{-1}).\end{equation}

 \item  Finally, for the variety $X = \Gm\backslash \PGL_2$,
 \begin{equation} \mu_X(\chi) = \gamma(\chi,-\frac{\check\alpha}{2}, \frac{1}{2}, \psi^{-1})^2 \cdot \gamma(\chi,\check\alpha, 0, \psi^{-1}).\end{equation}
\end{itemize}
\end{theorem}

The operator $\mathcal E_\emptyset^*$ plays a similar role as the transfer operator $\mathcal T$ that we are using in other comparisons, but we denote it by a different letter in order to avoid confusion when composing it with other transfer operators. Its geometric expression (i.e., the analog of \eqref{transfer-Cartan}) seems to be hard to describe, in general, and we will not proceed in this direction. We will instead investigate, with the help of the above analysis of relative characters, whether we can have commutative diagrams
\begin{equation}\label{Bcommute} \xymatrix{
\mathcal S(X\times X^\vee/G) \ar[r]^{\mathcal E_\emptyset^*}\ar[d]^{\mathcal T} & \mathcal S^+(X_\emptyset\times X_\emptyset^\vee/G) \ar[d]^{\mathcal T_\emptyset} \\
\mathcal S(Y\times Y^\vee/G') \ar[r]^{\mathcal E_\emptyset^*} & \mathcal S^+(Y_\emptyset\times Y_\emptyset^\vee/G') }\end{equation}
for the tranfer operators $\mathcal T$, $\mathcal T_\emptyset$ between the $\RTF$ quotients for two spherical varieties $X$ and $Y$ and their boundary degenerations, and whether there are similarities between the geometric expressions of $\mathcal T$ and $\mathcal T_\emptyset$. 


We assume that our spherical varieties $X$ and $Y$ belong to the list of varieties of Theorem \ref{thmpullbackfrombd}, and have the same dual group. In particular, the spaces $X_\emptyset\times X_\emptyset^\vee\sslash G$ and $Y_\emptyset\times Y_\emptyset^\vee\sslash G'$ both contain the same torus $A_X\simeq A_Y$ as an open subset (after fixing a base point, which we do as in the theorem).
To describe the ``correct'' transfer operators $\mathcal T_\emptyset$ that will make the diagram \eqref{Bcommute} commute, we should take into account the calculation of pull-backs of relative characters. Thus, in the end, the ``correct'' operators $\mathcal T_\emptyset$ do not depend just on the varieties $X_\emptyset$, $Y_\emptyset$ as abstract varieties, which might be isomorphic to each other, but also on the scattering morphisms that they are endowed with (which underlie the calculation of Theorem \ref{thmpullbackfrombd}).

More precisely, for the diagram \eqref{Bcommute} to commute, we would need the following to hold, using the scalars $\mu_X(\chi), \mu_Y(\chi)$ of Theorem \ref{thmpullbackfrombd}:

\begin{equation}\label{degentransferfe}
\mu_Y(\chi) \widecheck{(\mathcal T_\emptyset f)}(\chi\delta_{P(Y)}^{\frac{1}{2}}) = \mu_X(\chi) \check f(\chi\delta_{P(X)}^{\frac{1}{2}}). 
\end{equation}

Formally, at least, this means that the operator $\mathcal T_\emptyset$ is given as multiplicative convolution on $A_X$ by the measure $\nu_{X_\emptyset\to Y_\emptyset}$ whose Mellin transform is 
 \begin{equation}\label{degentransfermeasure}
  \check\nu_{X_\emptyset\to Y_\emptyset} (\chi) = \frac{\mu_Y(\chi \delta_{P(Y)}^{-\frac{1}{2}})}{\mu_X(\chi \delta_{P(X)}^{-\frac{1}{2}})}.
 \end{equation}

We will explicate these transfer operators case-by-case in the examples that follow.

\part{Special examples}

\section{Kuznetsov to trace formula for $\SL_2$} \label{sec:Rudnick}

Here we discuss the local comparison behind Rudnick's thesis \cite{Rudnick}. The goal is to compare stable orbital integrals for the trace formula and the Kunzetsov formula for the group $\SL_2$. Both the group $G=\SL_2$, and the Whittaker space $(N,\psi)\backslash G$ have the same dual group, namely, $\PGL_2$. So, we expect a local transfer map 
$$ \mathcal T: \mathcal S^-((N,\psi)\backslash G/(N,\psi)) \to \mathcal S(\frac{G}{G}),$$
which gives rise to stable functoriality between the Kuznetsov and the Selberg trace formula. The non-standard space $\mathcal S^-((N,\psi)\backslash G/(N,\psi)) $ of orbital integrals is the one corresponding to $L(\Ad, 1)$, as we will recall below. 

As before, we choose the identification $N\backslash G\sslash N = \Ga$ given by 
$$ \begin{pmatrix} a & b \\ c & d \end{pmatrix}\mapsto \zeta:= c,$$
and the section $ \zeta\mapsto \tilde\zeta =\left(\begin{array}{cc}
 & -\zeta^{-1} \\
\zeta &
\end{array}\right)$ over $\Gm$.	
We fix the chosen self-dual (with respect to $\psi$) Haar measure on $\Ga\simeq N$ and $\Ga \simeq N\backslash G\sslash N$; the integration formula \eqref{integration} now reads:
\begin{equation}\label{integration-KTF}
\int_G \Phi(g) dg = \int_{N\backslash G\sslash N} \int_{N\times N} \Phi(n_1 \tilde\zeta n_2) dn_1 dn_2 d\zeta,
\end{equation}
for a choice of Haar measure on $G$ that we fix from now on.

We will actually make no direct reference to Rudnick's thesis --- it is left to the reader to check that it can be reformulated in terms of the local comparison that we present here. Rather, we will start as in Section \ref{sec:Cartan}, by asking ourselves what is the correct pull-back $\mathcal T^*\Theta_\Pi$ of a tempered, stable character of $G$. It should be a relative character, or ``Bessel distribution'' (more correctly: Bessel generalized function) for the Kuznetsov formula, corrresponding to the generic representation $\pi$ in the packet denoted by $\Pi$. Recall that $L$-packets for $\SL_2$ consist of the irreducible components of restrictions of irreducible representations of  $\GL_2$, and that there exists a unique generic representation (for the given Whittaker datum $(N,\psi)$) in each tempered packet. The space of such Bessel distributions is one-dimensional, because this is the case for the space of morphisms
\begin{equation}\label{morphism}\mathcal S(N,\psi\backslash G)\hat\otimes\mathcal S(N,\psi^{-1}\backslash G) \to \pi \hat\otimes\tilde\pi;\end{equation}
the only question is what is the correct choice of scalar.

The answer is provided by global considerations, namely by the well-known relation between Petersson norms and Fourier coefficients of automorphic forms on $\GL_2$ (or $\SL_2$). More precisely, the following formula is known in this case for the square of the Whittaker period of a cusp form $\varphi \in \pi$, where $\pi$ denotes a generic automorphic representation of $\SL_2$:
\begin{equation}
 \left|\int_{[N]} \varphi(n) \psi(n) dn \right|^2 = \prod_v^* \int^*_{N} \left<\pi_v(n) \varphi_v, \varphi_v\right> \psi(n) dn.
\end{equation}
Here, we have fixed an isomorphism of abstract, unitary adelic representations: $\pi = \bigotimes_v' \pi_v$, and have assumed, accordingly, that the vector $\varphi\in \pi$ factorizes as a product of $\varphi_v$'s. The measure on $[\SL_2]$ used to define the unitary structure on $\pi$ is Tamagawa measure, while the $\pi_v$'s are abstract unitary representations of $G(k_v)$, with a distinguished spherical vector $\varphi_v^0$ at almost every place, satisfying $\Vert \varphi_v^0\Vert = 1$. In the formula above both the Euler product and the local integrals have to be understood in a regularized sense: 
The integral is understood as the value at $\lambda =1$ of the Fourier transform (defined as $\int \Phi(x)\psi(\lambda x) dx$) of the $L^2$-function $n\mapsto \left<\pi(n) \varphi_v, \varphi_v\right>$ (where $N$ is identified again with $\Ga$) --- I point the reader to \cite[\S 6.3]{SV} for details. The local factors of the Euler product are equal to the inverse adjoint unramified $L$-value $L(\pi_v, \Ad, 1)^{-1}$ at almost every place $v$, and their product has to be understood as the inverse of a partial $L$-value $L^S(\pi, \Ad,1)^{-1}$ for a large enough finite set $S$ of places, times the remaining factors. The validity of this formula follows from Rankin--Selberg theory --- I point the reader to \cite[Theorem 3.2]{LM} or \cite[\S 18.1]{SV} for details. The formula conjecturally generalizes to arbitrary groups, according to a conjecture of Lapid and Mao \cite{LM} which resembles the Ichino--Ikeda conjecture \cite{II}.

Returning to local notation, with $\pi$ an irreducible generic represenation of $G(F)$, it is natural to demand that the correctly normalized Bessel distribution $J_\pi$ (or $J_\Pi$, where $\Pi$ denotes the $L$-packet of $\pi$) be such that the adjoint of \eqref{morphism}, composed with evaluation on the cosets of $1$, is the $(N,\psi)\otimes (N,\psi^{-1})$-equivariant functional  
$$\tilde\pi\hat\otimes\pi\ni \tilde v \otimes v\mapsto \int^*_{N} \left<\pi(n) v, \tilde v\right> \psi(n) dn.$$ 
Explicitly, 
$$ J_\Pi(\Phi_1\otimes\Phi_2) = \sum_{(v,\tilde v)} \int_{(N\backslash G)^2} \Phi_1(x_1)\Phi_2(x_2) \int^*_{N} \left<\pi(nx_1) v, \tilde\pi(x_2)\tilde v\right> \psi(n) dn,$$
where $(v,\tilde v)$ runs over dual pairs in dual bases of $\pi$ and $\tilde\pi$.


We are now seeking the transfer operator from $\mathcal S((N,\psi)\backslash G/(N,\psi)) $ to $\mathcal S(\frac{G}{G})$ that will pull back the stable character $\Theta_\Pi$ to $J_\Pi$. If we believe that this operator should indeed have image in $\mathcal S(\frac{G}{G})$, which happens to be identified with the space of coinvariants of $\mathcal S(G)$ under the $\PGL_2$-adjoint action, then it is completely determined by this property, over all tempered packets $\Pi$, since stable tempered characters are dense in the space of stable tempered distributions. 

The following theorem will be proven in \cite{SaTransfer1, SaTransfer2}:

\begin{theorem}\label{thmRudnick}
Fix the isomorphism $N\backslash G\sslash N \simeq \Ga$ as before, thus identifying the subvariety $Y$ of anti-diagonal matrices used to trivialize orbital integrals (\S \ref{sstrivialization}) with $\Gm$. Fix the isomorphism $\Dfrac{G}{G} \simeq \Ga$ via the trace map. Consider the equivariant Fourier transform $\mathcal T:=\mathscr F_{\Id,1}$ of multiplicative convolution with the measure $D_1 = \psi(\zeta) d\zeta = \psi(\zeta) |\zeta|d^\times\zeta$ on $\Gm$.

Then:
\begin{enumerate}
 \item The convolution makes sense on $\mathcal S(N,\psi\backslash G/N,\psi)$ as the Fourier transform of a distribution, and maps it into $\mathcal S(\frac{G}{G})$. 
 \item For every tempered packet $\Pi$ and $J_\Pi$ as above, \begin{equation} \label{pullbackchar}
        \mathcal T^*\Theta_\Pi = J_\Pi.
       \end{equation} 
 \item The transform extends to an isomorphism, given by the same convolution understood, again, as the Fourier transform of a(n $L^2$-)distribution:
 \begin{equation}
  \mathcal T: \mathcal S_{L(\Ad,1)}^-(N,\psi\backslash G/N,\psi) \xrightarrow\sim \mathcal S(\frac{G}{G}),
 \end{equation}
 where the space on the left is the extended Schwartz space associated to the adjoint $L$-function at $1$, described in \S \ref{ssnonstandard}.
 \item At non-Archimedean places, it satisfies the fundamental lemma for the Hecke algebra, namely: for all $h\in \mathcal H(G,K)$, it takes the element
 $$ h\cdot f^0_{L(\Ad, 1)} \in \mathcal S_{L(\Ad,1)}^-(N,\psi\backslash G/N,\psi)$$ to the image of $h\cdot 1_{G(\mathfrak o)}dg$ in $\mathcal S(\frac{G}{G})$, for a suitable Haar measure $dg$. 
\end{enumerate}
\end{theorem}

The last statement of the proposition needs clarification: First of all, $K = \SL_2(\mathfrak o)$, where $\mathfrak o$ denotes the ring of integers of $F$. For a $G$-space $X$, the unramified Hecke algebra $\mathcal H(G,K)$ acts on $C^\infty(X)^K$. Since our basic vector $f^0_{L(\Ad, 1)}$ descends, by definition (\S \ref{ssnonstandard}), from an unramified Whittaker measure $\varphi$, we write $h\cdot f^0_{L(\Ad, 1)} $ for the image of $h\cdot \varphi$ in $\mathcal S^-_{L(\Ad, 1)} (N,\psi\backslash G/N,\psi)$. More intrinsically, the action of the unramified Hecke algebra on unramified vectors can be identified with the action of the ``unramified'' component of the Bernstein center (which is isomorphic to $\mathcal H(G,K)$), and this action descends to the coinvariant space 
$$ \mathcal S(N,\psi\backslash G/N,\psi) = \mathcal S((N,\psi\backslash G) \times (N,\psi^{-1}\backslash G))_{G^\diag}$$
(via its action on the first factor) and its extensions. Thus, it is not really an abuse of notation to write $h\cdot f^0_{L(\Ad, 1)} $, if we identify the unramified Hecke algebra with the unramified component of the Bernstein center. 

The statement on the action of the Hecke algebra is superfluous, here, because it can be inferred from the fundamental lemma for basic vectors, together with the statement on characters. In more general cases, though, it will be impossible to prove the statement on characters by a local argument. In those cases, the fundamental lemma for the full Hecke algebra will be the essential local input in order to separate a global identity of trace formulas into the contributions of individual $L$-packets and thus, via a global-to-local argument, prove the local transfer of relative characters.

What the reader should take away from the theorem above is that a local transfer operator characterized spectrally, in terms of relative characters, has a very simple geometric form --- essentially, a Fourier transform. This lies behind the global comparison of trace formulas in the thesis of Rudnick, who only considered holomorphic modular forms, and it should also be related to the Poisson summation formula between an approximation to the trace formula and an approximation to the Kuznetsov formula, proven in the work of Ali Altu\u{g} \cite{Altug1}. It should be possible to use the methods of \cite{SaBE2} to upgrade this to a full comparison of trace formulas, where on the Kuznetsov side, because of the insertion of $L$-functions, the relative trace formula needs to be understood using analytic continuation from $L(\Ad, 1+s)$ to $L(\Ad, 1)$ (again, as in \cite{SaBE2}).

Finally, we can compare the transfer operator $\mathcal T$ of the theorem above with the transfer opertor $\mathcal T_\emptyset$ for the boundary degenerations of the group and the Whittaker space. Recall that, for $Y=\SL_2$, we have $Y_\emptyset = $ the variety of $2\times 2$ matrices of rank one, while for $X=$ the Whittaker model we have $X_\emptyset=$ the same space $N\backslash G$ with trivial character on $N$. Equation \eqref{degentransferfe}
together with Theorem \ref{thmpullbackfrombd} in this case read:
$$ \check f(\chi) = \gamma(\chi\delta^{-\frac{1}{2}}, -\check\alpha, 0, \psi^{-1}) \widecheck{\mathcal T_\emptyset f} (\chi)= \gamma(\chi,-\check\alpha,1, \psi^{-1}) \widecheck{\mathcal T_\emptyset f} (\chi),$$
and Proposition \ref{gamma} implies that the transfer operator $\mathcal T_\emptyset$ in this case should be given by $\mathscr F_{\check\alpha, 1}$. (Note that $\gamma(\chi,-\check\alpha,1, \psi^{-1})^{-1} = \gamma(\chi,\check\alpha,0, \psi)$.)

We notice that, for the coordinates on $N\backslash G\sslash N$ and $\Dfrac{G}{G}$ fixed here, the transfer operator $\mathcal T$ is exactly the same as the transfer operator $\mathcal T_\emptyset$, namely, convolution by the measure $D_1$!

\section{Kuznetsov to relative trace formula for toric periods} \label{sec:Waldspurger}

Now consider the case of $\X=T\backslash G/T$, where $G=\PGL_2$ and $T\simeq \Gm$ is a split torus. One could also consider a non-split torus, but would need to slightly modify the equivariant Fourier transforms that we presented in \S \ref{sec:Fourier}. 

The role of the character $\Theta_\Pi$, here, will be played by a relative character $I_\pi$ for an irreducible tempered representation of $\PGL_2$, for the quotient space $\Gm\backslash \PGL_2/\Gm$. The definition of the relative character $I_\pi$ is completely analogous to that of the Kuznetsov relative character $J_\pi$: It is given as the composition
$$ \mathcal S(X\times X)\to \pi\hat\otimes\tilde\pi\to\CC$$
(where $X=T\backslash G$),
where the dual of the map to $\pi\hat\otimes\tilde\pi$ is the morphism
$$ \tilde\pi\otimes\pi\to C^\infty(X\times X)$$
that, composed with evaluation at $T1 \times T1$ is given by:
$$\tilde v \otimes v\mapsto \int_{T} \left<\pi(t) v, \tilde v\right> dt.$$ 
Here the integral is convergent (for tempered representations), and no normalization is needed. Moreover, the $L$-packets for the group $\PGL_2$ are singletons (if we do not consider its inner forms, which we should have, in the case of a non-split torus), therefore there is no need to distinguish, notationally, between $\pi$ and its $L$-packet $\Pi$.

The following theorem is a reformulation of results proven in \cite{SaBE1, SaBE2}.

\begin{theorem}
Fix the isomorphism $N\backslash G\sslash N \simeq \Ga $ as before, and an isomorphism $T\backslash G\sslash T = \Ga$ such that the identity element of $G$ maps to $1 \in \Ga$, and the non-trivial element of the Weyl group of $T$ maps to $0$. 

Consider the equivariant Fourier transform $\mathcal T:=\mathscr F_{\Id,1} \circ \mathscr F_{\Id,1}$ of multiplicative convolution, twice, with $D_{1} = \psi(\bullet) |\bullet| d^\times\bullet$ on measures on $\Gm$.

Then:
\begin{enumerate}
 \item The convolution makes sense on $\mathcal S(N,\psi\backslash G/N,\psi)$ as the Fourier transform of a distribution, and maps it into $\mathcal S(T\backslash G/T)$. 
  \item For every tempered representation $\pi$, \begin{equation} 
        \mathcal T^*I_\pi = J_\pi.
       \end{equation} 
 \item The transform extends to an isomorphism, given by the same convolution understood, again, as the Fourier transform of a(n $L^2$-)distribution:
 \begin{equation}
  \mathcal T: \mathcal S_{L(\Std,\frac{1}{2})^2}^-(N,\psi\backslash G/N,\psi) \xrightarrow\sim \mathcal S(T\backslash G/T).
 \end{equation}
 \item At non-Archimedean places, it satisfies the fundamental lemma for the Hecke algebra, namely: for all $h\in \mathcal H(G,K)$, it takes the element
 $$ h\cdot f^0_{L(\Std, \frac{1}{2})^2} \in \mathcal S_{L(\Std,\frac{1}{2})^2}^-(N,\psi\backslash G/N,\psi)$$ to the image of $h\cdot 1_{T\backslash G(\mathfrak o)}$ in $\mathcal S(T\backslash G/T)$.
\end{enumerate}
\end{theorem}

\begin{proof}
The matching of the two spaces is \cite[Theorem 5.1]{SaBE1}, and the fundamental lemma for the Hecke algebra is \cite[Theorem 5.4]{SaBE1}. Note some differences in the coordinates used here and there: our coordinate $\xi$ for the Kuznetsov formula is $-\xi^{-1}$ there, and our isomorphism $T\backslash G\sslash T = \Ga$ is the negative of the isomorphism of \emph{loc.cit.} If we take these changes of coordinates into account, the transfer operator $|\bullet| \mathcal G$ of \cite[Theorem 5.1]{SaBE1}, from orbital integrals for $T\backslash G/T$ to orbital integrals for the Kuznetsov formula, becomes twice convolution with the measure $\psi^{-1}(\frac{1}{x}) d(\frac{1}{x})$ on $F^\times$; the inverse of that is twice convolution with $\psi(x) d^\times x$. This is the transfer operator between the \emph{functions} of orbital integrals, whereas (in our current coordinates) the Schwartz measures for both $T\backslash G/T$ and $(N,\psi)\backslash G/(N,\psi)$ are the product of those functions with an additive Haar measure on $\Ga$. Thus, the transfer operator $\mathcal T$ from test measures for the Kuznetsov formula to test measures for $T\backslash G/T$ is given by twice convolution by $\psi(x) |x| d^\times x$.

The statement on relative characters is \cite[Theorem 7.1.3]{SaBE2}.
\end{proof}

Again, we can compare the transfer operator $\mathcal T$ of the theorem above with the transfer operator $\mathcal T_\emptyset$ for the boundary degenerations of the space $Y=\Gm\backslash \PGL_2$ and the Whittaker space. In both cases, here, $X_\emptyset \simeq Y_\emptyset \simeq N\backslash G$ (but endowed with different scattering operators).  Equation \eqref{degentransfermeasure}, 
together with Theorem \ref{thmpullbackfrombd} in this case read:
$$ \check f(\chi) = \gamma(\chi\delta^{-\frac{1}{2}}, -\frac{\check\alpha}{2}, \frac{1}{2}, \psi^{-1})^2 \cdot  \widecheck{\mathcal T_\emptyset f} (\chi)= \gamma(\chi,-\frac{\check\alpha}{2},1, \psi^{-1})^2 \cdot \widecheck{\mathcal T_\emptyset f} (\chi),$$
and Proposition \ref{gamma} implies that the transfer operator $\mathcal T_\emptyset$ in this case should be given by $\mathscr F_{\frac{\check\alpha}{2}, 1}\circ \mathscr F_{\frac{\check\alpha}{2}, 1}$. 

Again, for the coordinates on $N\backslash G\sslash N$ and $T\backslash G\sslash T$ fixed here, the transfer operator $\mathcal T$ is exactly the same as the transfer operator $\mathcal T_\emptyset$!

\section{Functional equation for the standard $L$-function of $\GL_n$} \label{sec:standard}

Unlike the case of transfer operators between different relative trace formulas, for the functional equations of $L$-functions it turns out that we get more natural formulas working with \emph{half-densities}. 

Let $G=\GL_n$. The theorem that follows is due to Jacquet \cite[Theorem 1]{Jacquet}. To state it, recall that we have identified an open subset of $N\backslash G\sslash N$ with the universal Cartan $A$, by identifying the latter with the torus $T$ of diagonal matrices (through the upper triangular Borel), embedded as $T\mapsto wT \to N\backslash G\sslash N$. Here, choose $w=w_n=$ the anti-diagonal matrix whose entries on the anti-diagonal are all $1$.

\begin{theorem}\label{thmJacquet}
Let $G=\GL_n$. Consider the diagram
 $$\xymatrix{
 \mathcal D(\Mat_n) \ar[d]\ar[r]^{\mathcal F}& \mathcal D(\Mat_n^\vee)\ar[d]\\
 \mathcal D^-_{L(\Std,\frac{1}{2})}(N,\psi\backslash G/N,\psi) \ar@{-->}[r]^{\mathcal H_{\Std}} & \mathcal D^-_{L(\Std^\vee,\frac{1}{2})}(N,\psi\backslash G/N,\psi)
 }$$
where $\mathcal F$ denotes the equivariant Fourier transform: 
$$ \mathcal F(\varphi) (y) = \left(\int_{\Mat_n} \varphi(x) \psi(\left<x, y\right>) dx^\frac{1}{2}\right) dy^\frac{1}{2}$$
(for dual Haar measures $dx$, $dy$ on $\Mat_n$ and $\Mat_n^\vee$ with respect to the character $\psi$).

 There is a linear isomorphism $\mathcal H_{\Std}$ as above, 
 making the diagram commute. Moreover, $\mathcal H_{\Std}$ is given by the following formula:
 
 \begin{equation}
  \mathcal H_{\Std} = \mathscr F_{-\check\epsilon_1,\frac{1}{2}} \circ \psi(-e^{-\alpha_1}) \circ \mathscr F_{-\check\epsilon_2,\frac{1}{2}} \circ \cdots \circ \psi(-e^{-\alpha_{n-1}}) \circ \mathscr F_{-\check\epsilon_n, \frac{1}{2}},
 \end{equation}
where by $\check\epsilon_i$ we denote the cocharacter of the $i$-th coordinate of the torus of diagonal elements (written additively), identified with the universal Cartan via the Borel of upper triangular elements (and hence the cocharacter into the $i$-th column of the anti-diagonal, when we identify it with this torus via $a\mapsto w_n a$). The intervening factors $\psi(-e^{-\alpha_i})$ denote multiplication by $\psi$ composed with the minus the value of the indicated root (denoted exponentially here, to avoid confusion with the additive notation for weights).\footnote{For example, for $G=\GL_2$, $\psi(-e^{-\alpha_1})$ denotes the function $\begin{pmatrix} a & * \\ &d \end{pmatrix}\mapsto \psi(-\frac{d}{a})$. }
 \end{theorem}
 
 \begin{remark}
  Notice that the transform $\mathcal H_{\Std}$ by construction satisfies the fundamental lemma for the Hecke algebra: the basic vector\footnote{We use the same symbol for the ``half-density'' basic vector as for the ``measure'' basic vector; their quotient is simply a half-density $(\delta(t) dt)^\frac{1}{2}$ on $N\backslash G\sslash N$, according to the integration formula \eqref{integration}. (Remember that $dt$ denotes a Haar measure on the torus $A$, here.)} $f^0_{L(\Std,\frac{1}{2})}$ maps to the basic vector $f^0_{L(\Std^\vee,\frac{1}{2})}$, and the same holds for their Hecke translates. Indeed, the basic vectors are push-forwards of the characteristic functions of $\Mat_n(\mathfrak o)$, resp.\ $\Mat_n(\mathfrak o)$, times a Haar half-density on the additive group of matrices, and Fourier transform takes one to the other, and is equivariant with respect to the $G$-action (in particular, the action of the Hecke algebra).
 \end{remark}

\begin{proof}
 Jacquet's Theorem 1 in \cite{Jacquet} is stated for the functions 
 $$a\mapsto |\det(a)|^\frac{n-1}{2} \delta^\frac{1}{2}(a) O_a(\Phi),$$
 where $O_a$ is the orbital integral represented by $a$, and $\Phi\in \mathscr F(\Mat_n)$, the Schwartz space of \emph{functions} on $\Mat_n$. To compare it with our setup, where a density on $\Mat_n$ is of the form $|\det|^{\frac{n}{2}}\Phi$ (times a Haar half-density on $G$) and its push-forward as an element of $\mathcal D^-_{L(\Std,\frac{1}{2})}(N,\psi\backslash G/N,\psi) $ is equal to $f:=|\det(p)|^\frac{n}{2} \delta^\frac{1}{2}(p) O_p(\Phi)$ (times a Haar half-density on the torus), the measures $dp_1 \cdots dp_n$ in Jacquet's formula have to be multiplied by $|p_1 \cdots p_n|^{-\frac{1}{2}}$. Moreover, Jacquet's result is in terms of an endomorphism of $\mathscr F(\Mat_n)$ that he denotes by $\check\Phi$. To make this into our equivariant transform $\mathcal F$ from $\mathcal D(\Mat_n)$ to $\mathcal D(\Mat_n^\vee)$, we must set 
 $$ \mathcal F(|\det|^{\frac{n}{2}}\Phi)(g)  = |\det(g)|^{-\frac{n}{2}} \check\Phi(w_n g^{-1} w_n) = |\det(g)|^{-\frac{n}{2}} \hat \Phi(g^{-1})$$
(using Jacquet's notation --- compare with \eqref{pushf-Matn}), and replace $\psi$ by $\psi^{-1}$ in Jacquet's definition of $\hat\Phi$. The push-forward of this to $\mathcal D^-_{L(\Std^\vee,\frac{1}{2})}(N,\psi\backslash G/N,\psi) $ is 
$$ \mathcal H_{\Std} f = |\det(b)|^{-\frac{n}{2}} \delta^\frac{1}{2}(b) O_{w_n b w_n}(\check\Phi)$$
(times the same Haar half-density on the torus as before).

Thus, if our coordinate is an anti-diagonal element $b$ with entry $b_i$ on the $i$-th column, we must set $a_i = b_{n+1-i}^{-1}$ in Jacquet's formula, multiply the final result by $|b_1 \cdots b_n|^{-\frac{1}{2}}$, and replace $\psi$ by $\psi^{-1}$. It then reads:

$$\mathcal H_{\Std} f (b_1, \dots, b_n)  = $$
$$= \int f(b_1 p_1, \dots, b_n p_n) \psi(\sum_{i=1}^n p_i - \sum_{i=1}^{n-1} \frac{b_{i+1}}{p_i b_i}) |p_1 \cdots p_n|^\frac{1}{2} d^\times p_1 \cdots d^\times p_n,$$
which is the formula of the theorem.

\end{proof}

The theorem implies a statement about pull-backs of relative characters via Hankel transforms. Namely, let $J_\pi$ denote the Kuznetsov relative character of an irreducible representation $\pi$ of $G$, defined exactly as (for $\SL_2$) in Section \ref{sec:Rudnick}, but now considered as a generalized half-density by multiplying by a half-density of the form $\delta(t)^\frac{1}{2} (dt)^\frac{1}{2}$ (cf.\ \eqref{pushf-densities}). Its pull-back (dual to the twisted push-forward $p_!$) to $G$ is a generalized matrix coefficient for $\pi$, thought of as a functional on half-densities on $G$. As $\pi$ varies in the family $\{\pi \otimes |\det|^s\}_{s\in \CC}$, we can extend $J_\pi$ to a meromorphic family of functionals on $\mathcal D(\Mat_n)$, first by a convergent integral (when $\Re(s)\gg 0$), and then by meromorphic continuation. We can similarly consider it as a meromorphic family of functionals on $\mathcal D(\Mat_n^\vee)$.

The adjoint of equivariant  transform $\mathcal F$ acts by the scalar $\gamma(\pi, \Std, \frac{1}{2}, \psi)$ on generalized matrix coefficients of $\pi$, and thus we get:

\begin{corollary}\label{corcharsHankel}
We have  
\begin{equation}\mathcal H_{\Std}^* J_\pi = \gamma(\pi, \Std, \frac{1}{2}, \psi) \cdot J_\pi,\end{equation}
as meromorphic families of functionals on $\mathcal D^-_{L(\Std,\frac{1}{2})}(N,\psi\backslash G/N,\psi)$.
\end{corollary}

We notice, again, that the operator (which here we call ``Hankel transform'') has a very simple form, given as a composition of equivariant Fourier transforms, with some intervening scalar factors. Globally, these scalar factors are equal to $1$ at every point of $A(k)$. Thus, it seems plausible that one could prove an equality of Kuznetsov formulas, one applied to a half-density $ f\in \mathcal D^-_{L(\Std,\frac{1}{2})}(N,\psi\backslash G/N,\psi)$ and the other applied to the half-density $\mathcal H_{\Std} f \in \mathcal D^-_{L(\Std^\vee,\frac{1}{2})}(N,\psi\backslash G/N,\psi)$. In fact, the reader can check that the operator of the functional equation for the \emph{square} of the standard $L$-function, introduced in \cite[Proposition 6.5.2]{SaBE2} and denoted by $\mathcal T$ there, is the square of our Hankel transform $\mathcal H_{\Std}$; this operator was used in \emph{loc.cit}.\ to give a purely trace formula-theoretic proof of the functional equation of the square of the standard $L$-function. I also point the reader to Herman's paper \cite{Herman} which, although not written in local language, gives a similar proof for the functional equation of the standard $L$-function.

Finally, taking boundary degenerations into consideration, we can examine whether there is a commutative diagram 
\begin{equation} \xymatrix{
 \mathcal D^-_{L(\Std,\frac{1}{2})}(N,\psi\backslash G/N,\psi)  \ar[r]^{\mathcal E_\emptyset^*}\ar[d]^{\mathcal H_\Std} & \mathcal D^\pm_{L(\Std,\frac{1}{2})}(N\backslash G/N) \ar[d]^{\mathcal H_{\Std,\emptyset}} \\
\mathcal D^-_{L(\Std^\vee,\frac{1}{2})}(N,\psi\backslash G/N,\psi) \ar[r]^{\mathcal E_\emptyset^*} & \mathcal D^\pm_{L(\Std^\vee,\frac{1}{2})}(N\backslash G/N) }\end{equation}
for some suitable operator $\mathcal H_{\Std,\emptyset}$. The spaces on the right are obtained by applying the asymptotics morphisms on the spaces on the left --- we will refrain from examining these spaces here. It will suffice, for here, to observe that, for the diagram to commute, the operator $\mathcal H_{\Std,\emptyset}$ should act on Mellin transforms by a scalar, which is the restriction of the gamma factor of the standard $L$-function, via the embedding 
\begin{equation}\label{jembedding}j:\check A\hookrightarrow \check G\end{equation}
to the characters of $A$. The precise calculation is as follows: 

We identify again an open subset of $N\backslash G/N$ with the torus $A$ as before (by fixing a base point), and define Mellin transform of half-densities as
$$ \check f(\chi) = \int_A f(t) \chi^{-1} (t) (dt)^\frac{1}{2},$$
where $dt$ is a fixed Haar measure on $A$. Then the operator $\mathcal H_{\Std,\emptyset}$ should satisfy:
$$ \widecheck{(\mathcal H_{\Std,\emptyset} f)} (\chi) = \gamma(\chi^{-1},\Std\circ j,  \frac{1}{2}, \psi) \check f(\chi) = \prod_{i=1}^n \gamma(\chi, -\check\epsilon_i , \frac{1}{2}, \psi) \check f(\chi),$$ 
where $j$ is the embedding of \eqref{jembedding}.
(The choice of additive character $\psi$, here, matches the Godement--Jacquet local functional equation for the Fourier transform $\mathcal F$.)

By Proposition \ref{gamma}, we have $\mathcal H_{\Std,\emptyset}  = \prod_{i=1}^n \mathscr F_{-\check\epsilon_i, \frac{1}{2}}$. We notice that the operator $\mathcal H_{\Std}$ looks just like a \emph{deformation} of $\mathcal H_{\Std,\emptyset}  $, with the intermediate scalar factors $\psi(-e^{-\alpha_i})$ inserted!

\section{Functional equation for the symmetric square $L$-function of $\GL_2$} \label{sec:sym2}

The representation $\Sym^2$ of $\GL_2$ factors through $\PGL_2 \times \Gm$, the dual group of $G=\SL_2 \times \Gm$. We wish to study the functional equation for the $L$-function associated to $\Sym^2$ at the level of the Kuznetsov formula of $G$. 

We denote by $\check\lambda_+, \check\lambda_0$ and $\check\lambda_-$ the three weights of $\Sym^2$, so that $\check\lambda_-$ is anti-dominant and $\check\lambda_+$ is dominant. 

The analog of Theorem \ref{thmJacquet} and Corollary \ref{corcharsHankel} in this setting is the following. To express it, we denote by $\eta_{D}$ the quadratic character associated to the quadratic extension $F(\sqrt{D})$, considered as a character of the $\Gm$-factor, and identified with the operator of multiplication by this character. We denote by $\delta_a$ the operator of multiplicative translation by $a$, under the $\Gm$-action on $N\backslash G\sslash N$. The letter $\zeta$ denotes the coordinate on $N\backslash \SL_2\sslash N$ (and hence also on $N\backslash G\sslash N$, by projection) that was fixed in \S \ref{ssnonstandard}. 
For a half-density $\varphi$ on $N\backslash G\sslash N$, its ``$(\Gm,|\bullet|^s)$-equivariant integral'' is the following half-density on $N\backslash \SL_2\sslash N$:
$$ \zeta\mapsto \int_{\Gm} \varphi(\zeta, a) |a|^{-s} (d^\times a)^\frac{1}{2},$$
where $(d^\times a)^\frac{1}{2}$ is a fixed Haar half-density on $F^\times$.  Finally, for a quadratic character $\eta$ with associated quadratic extension $E$, 
$\lambda(\eta,\psi)$ denotes the ratio of abelian gamma-factors:
\begin{equation}\label{lambda} \lambda(\eta,\psi) = \frac{\gamma(1, s, \psi)\gamma(\eta, s,\psi)}{\gamma_E(1,s, \psi\circ\tr)}.
\end{equation}

\begin{theorem}
There is a space $\mathcal D^-_{L(\Sym^2, \frac{1}{2})} (N,\psi\backslash G /N,\psi)$ of (densely defined) half-densities on $N\backslash G\sslash N$, containing $\mathcal D (N,\psi\backslash G /N,\psi)$, whose $(\Gm,|\bullet|^s)$-equivariant integrals converge for $\Re(s)\gg 0$, admit meromorphic continuation to the entire complex plane, and have image equal to $\mathcal D^-_{L(\Ad, \frac{1}{2}+ s)} (N,\psi\backslash \SL_2/N,\psi)$, for every $s$ away from the poles. 

Moreover the transform:
  \begin{equation}\label{HSymformula} 
  \mathcal H_{\Sym^2} = \lambda(\eta_{\zeta^2-4},\psi)^{-1} \mathscr F_{-\check\lambda_+,\frac{1}{2}} \circ \delta_{1-4\zeta^{-2}} \circ \eta_{\zeta^2-4} \circ \mathscr F_{-\check\lambda_0,\frac{1}{2}} \circ \eta_{\zeta^2-4} \circ \mathscr F_{-\check\lambda_-,\frac{1}{2}}\end{equation}
is a $\Gm$-equivariant isomorphism 
$$\mathcal D^-_{L(\Sym^2, \frac{1}{2})} (N,\psi\backslash G /N,\psi )\xrightarrow\sim \mathcal D^-_{L((\Sym^2)^\vee, \frac{1}{2})} (N,\psi\backslash G /N,\psi )$$
(where the space on the right is the analogous space with the $\Gm$-coordinate inverted), and hence descends for every $s\in \CC$ away from the poles to an isomorphism
$$\mathcal H_{\Ad,s} : \mathcal D^-_{L(\Ad, \frac{1}{2}+ s)} (N,\psi\backslash \SL_2/N,\psi) \xrightarrow\sim \mathcal D^-_{L(\Ad, \frac{1}{2}- s)} (N,\psi\backslash \SL_2/N,\psi).$$

As $\pi$ varies in a family of representations twisted by the character $|\bullet|^s$ of $\Gm$, the Hankel transform satisfies:
\begin{equation}\mathcal H_{\Sym^2}^* J_\pi = \gamma(\pi, \Sym^2, \frac{1}{2}, \psi) \cdot J_\pi,\end{equation}
as meromorphic families of functionals on $\mathcal D^-_{L(\Sym^2,\frac{1}{2})}(N,\psi\backslash G/N,\psi)$, where $J_\pi$ are the relative characters for the Kuznetsov formula, as before.

Finally, the space $\mathcal D^-_{L(\Sym^2, \frac{1}{2})} (N,\psi\backslash G /N,\psi)$ has a ``basic vector'' $f_{L(\Sym^2, \frac{1}{2})}^0$, the image of an unramified half-density on $(N,\psi\backslash G)$, whose images in all of the spaces  $\mathcal D^-_{L(\Ad, \frac{1}{2}+ s)} (N,\psi\backslash \SL_2/N,\psi)$ are equal to the basic vectors $f^0_{L(\Ad, \frac{1}{2}+ s)}$, and the Hankel transform $\mathcal H_{\Sym^2}$ maps $h\cdot f_{L(\Sym^2, \frac{1}{2})}^0$ to $h\cdot f_{L((\Sym^2)^\vee, \frac{1}{2})}^0$, for any element $h$ of the unramified Hecke algebra.
\end{theorem}

Again, the Hankel transform is expressed in terms of equivariant Fourier tranforms and relatively innocuous intermediate factors and operators which, in principle, preserve Poisson summation globally. Of course, the nature of these operators and the possibility of extending such formulas to higher symmetric powers remains unclear. But we can again view them as deformations of Hankel transforms for the boundary degenerations, namely, if we consider again a commutative diagram 
\begin{equation} \xymatrix{
 \mathcal D^-_{L(\Sym^2,\frac{1}{2})}(N,\psi\backslash G/N,\psi)  \ar[r]^{\mathcal E_\emptyset^*}\ar[d]^{\mathcal H_{\Sym^2}} & \mathcal D^\pm_{L(\Sym^2,\frac{1}{2})}(N\backslash G/N) \ar[d]^{\mathcal H_{\Sym^2,\emptyset}} \\
\mathcal D^-_{L((\Sym^2)^\vee,\frac{1}{2})}(N,\psi\backslash G/N,\psi) \ar[r]^{\mathcal E_\emptyset^*} & \mathcal D^\pm_{L((\Sym^2)^\vee,\frac{1}{2})}(N\backslash G/N) },
\end{equation}
we will find that the operator $\mathcal H_{\Sym^2,\emptyset}$ should satisfy:

$$ \widecheck{(\mathcal H_{\Sym^2,\emptyset} f)}(\chi) = \gamma( \chi^{-1}, \Sym^2\circ j,\frac{1}{2}, \psi) \check f(\chi) = \prod_{i=+, 0, -} \gamma(\chi, -\check\lambda_i , \frac{1}{2}, \psi) \check f(\chi),$$
and therefore, by Proposition \ref{gamma}:
$$\mathcal H_{\Sym^2,\emptyset} =  \prod_{i=+, 0, -} \mathscr F_{-\check\lambda_i, \frac{1}{2}}.$$

Thus, we observe again that the operator $\mathcal H_{\Sym^2}$ is just a deformation of $\mathcal H_{\Sym^2,\emptyset}$ with some intermediate operators which, again, in principle, would be suitable for a global Poisson summation formula!

\bibliographystyle{alphaurl}
\bibliography{biblio}

\end{document}